\documentclass{article}
\usepackage{geometry,amsmath,amssymb,enumerate,cite,comment,graphicx,caption,cite}
\usepackage[title]{appendix}
\usepackage{hyperref}
\usepackage[thmmarks,amsmath,thref]{ntheorem}
\newtheorem{theorem}{Theorem}[section]

\newtheorem{lemma}{Lemma}[section]
\newtheorem{assumption}{Assumption}[section]
\newtheorem{corollary}{Corollary}[section]
\theorembodyfont{\normalfont}
\newtheorem*{remark}{Remark}
\theoremsymbol{{$\square$}}
\newtheorem*{proof}{Proof}
\renewcommand{\d}{\mathrm{d}}

\newcommand{\W}{\mathcal{W}}
\newcommand{\Ge}{\geqslant}
\newcommand{\Le}{\leqslant}
\renewcommand{\P}{\mathcal{P}}
\newcommand{\D}{\mathcal{D}}
\newcommand{\C}{\mathcal{C}}
\newcommand{\E}{\mathbb{E}}

\newcommand{\di}{\displaystyle}
\numberwithin{equation}{section}
\title{Ergodicity and long-time behavior of the Random Batch Method for interacting particle systems}
\newcommand{\ShangJiao}{School of Mathematical Sciences, Institute of Natural Sciences, MOE-LSC, Shanghai Jiao Tong University, Shanghai, 200240, P. R. China. }
\newcommand{\BeiDa}{Beijing International Center for Mathematical Research, Peking University, Beijing, 100871, P. R. China. }
\author{%
Shi Jin\thanks{\ShangJiao Email: shijin-m@sjtu.edu.cn}
\and
Lei Li\thanks{School of Mathematical Sciences, Institute of Natural Sciences, Qing Yuan Research Institute, MOE-LSC, Shanghai Jiao Tong University, Shanghai, 200240, P. R. China. Email: leili2010@sjtu.edu.cn}
\and
Xuda Ye\thanks{\BeiDa Email: abneryepku@pku.edu.cn}
\and
Zhennan Zhou\thanks{\BeiDa Email: zhennan@bicmr.pku.edu.cn}
}
\begin{document}
\maketitle
\begin{abstract}
We study the geometric ergodicity and the long time behavior of the Random Batch Method for interacting particle systems, which exhibits superior numerical performance in recent large-scale scientific computing experiments. We show that for both the interacting particle system (IPS) and the random batch interacting particle system (RB--IPS), the distribution laws converge to their respective invariant distributions exponentially, and the convergence rate does not depend on the number of particles $N$, the time step $\tau$ for batch divisions or the batch size $p$. Moreover, the Wasserstein distance between the invariant distributions of the IPS and the RB--IPS is bounded by $O(\sqrt{\tau})$, showing that the RB--IPS can be used to sample the invariant distribution of the IPS accurately  with greatly reduced computational cost.
\end{abstract}
\begin{small}
\end{small}
\\[6pt]
\begin{small}
\end{small}
\section{Introduction}
\label{section:introduction}
Simulation of large-size dynamical systems has always been a computational bottleneck in optimization and stochastic sampling. One of the main difficulties is that the complexity of updating in a single time step is extremely high, which is often beyond the linear scaling with respect to the size of the system. In the past years, various approximate simulation methods have been developed to reduce the computational cost with tolerable numerical error, for example, the stochastic gradient descent (SGD)\cite{SGD_original} and the stochastic gradient Langevin dynamics (SGLD)\cite{SGLD_original}. 
These methods have been widely used in machine learning for efficient simulation, and one may refer to \cite{convergence_SGD_1,convergence_SGD_2,convergence_SGD_3,convergence_SGLD_1,convergence_SGLD_2} for the error analysis.

In this work, we focus on the interacting particle system (IPS), which is of vital importance in computational physics\cite{compute_physics_1,compute_physics_2} and computational chemistry\cite{frenkel2001,lelievre2016}. The study of their  mean-field limits have also been of significant recent research interest\cite{mean_field_1,mean_field_2,mean_field_3,JabinWang}.
Consider a system of $N$ particles represented by a collection of position variables $X_t = \{X_t^i\}_{i=1}^N$ with the position of each particle $X_t^i\in\mathbb R^{d}$, and the system of particles is evolved by the overdamped Langevin dynamics:
\begin{equation}
	\d X_t^i = b(X_t^i)\d t + \frac1{N-1}
	\sum_{j\neq i} K(X_t^i - X_t^j)\d t + \sigma\d W_t^i,~~~~
	i=1,\cdots,N.
	\label{IPS}
\end{equation}
Here, $b(\cdot):\mathbb R^d\rightarrow\mathbb R^d$ is the drift force, $K(\cdot):\mathbb R^d\rightarrow\mathbb R^d$ is the interaction force, $\sigma>0$ is a scalar constant, and $W_t = \{W_t^i\}_{i=1}^N$ are $N$ independent standard Wiener processes in $\mathbb R^d$.

With certain additional assumptions on the parameters, there exists an invariant distribution $\mu$ in $\mathbb R^{Nd}$ associated with the IPS \eqref{IPS}, and thus \eqref{IPS} can be utilized to produce samples of $\mu$ by time integration.
If the drift force $b(x) = -\nabla U(x)$ and the interaction force $K(x) = -\nabla V(x)$ for
some potential functions $U(x),V(x)$ with $\sigma = \sqrt{2}$ and $V(x)$ being even, then the invariant distribution $\mu$ can be explicitly expressed as
\begin{equation}
	\mu(\d x) \propto \exp\bigg(
		-\sum_{i=1}^N U(x^i) - 
		\frac1{N-1}\sum_{1\Le i<j\Le N}
		V(x^i-x^j)
	\bigg)\d x.
	\label{invariant_mu}
\end{equation}
To simulate the IPS \eqref{IPS} numerically, one has to discretize \eqref{IPS} in time and applies numerical integration in each time step. For an IPS of $N$ particles, it requires $O(N^2)$ complexity to compute all the interaction forces $\{K(X_t^i - X_t^j)\}_{i\neq j}$, hence the computational cost per time step is $O(N^2)$, which results in inefficiency of the simulation. Therefore, it is desirable to apply an approximate simulation method which is able to reduce the computational cost and still produce reliable samples of the invariant distribution $\mu$.

The Random Batch Method (RBM) proposed in \cite{RBM_original} is a simple random algorithm to reduce the computational cost from $O(N^2)$ to $O(N)$ in numerically evolving the IPS \eqref{IPS}.
As supported by extensive numerical tests\cite{RBM_sample_1,RBM_sample_2,RBM_sample_3}, the RBM is not only an efficient algorithm for the evolution of the system, it can also serve as an efficient simulation tool which preserves $\mu$ as its invariant distribution in an approximate sense, thus can be used as a sampling algorithm to obtain statistical samples of the invariant  measure of the IPS \eqref{IPS}. Yet, theoretical justification for the sampling accuracy is still lacking.

The idea of the RBM is illustrated as follows. Let $\tau>0$ be the time step for batch division and define $t_n :=n\tau$. 
For each $n\Ge0$, let the index set $\{1,\cdots,N\}$ be randomly divided into $q$ batches $\D=\{\C_1,\cdots,\C_q\}$, where each batch $\C\in\D$ has size $p = N/q$. The IPS \eqref{IPS} within the time interval $t\in[t_n,t_{n+1})$ is approximated as the SDE of $\tilde X_t = \{\tilde X_t\}_{i=1}^N$ in $\mathbb R^{Nd}$, given by
\begin{equation}
	\d \tilde X_t^i = b(\tilde X_t^i)\d t + 
	\frac1{p-1}
	\sum_{j\neq i,j\in\C}
	K(\tilde X_t^i - \tilde X_t^j)\d t + 
	\sigma\d W_t^i,~~
	i\in \C,~~t\in[t_n,t_{n+1}),
	\label{RBM_IPS}
\end{equation}
where $\C\in\D$ is the batch which contains $i$.
For the next time interval, the previous division $\D$ is discarded and another random division $\D'$ is employed to form the dynamics \eqref{RBM_IPS}.
We also point out that the RBM is not only a numerical method for the IPS \eqref{IPS}. It is also a stochastic model for interacting particle systems, in which particles interact, within each time interval of length $\tau$, with a small number ($p-1$) of particles.
In the following, the dynamical system \eqref{RBM_IPS} will be referred to as the random batch interacting particle system (RB--IPS), as a comparison to the IPS \eqref{IPS}. For the convenience of analysis, assume both \eqref{IPS}\eqref{RBM_IPS} are exactly integrated in time, thus there is no error due to numerical discretization for the time derivative.

If one numerically integrates \eqref{IPS}\eqref{RBM_IPS} in each time step, the RB--IPS is able to reduce the computational cost per time step from $O(N^2)$ to $O(Np)$,
because one only needs to compute the interaction forces within each batch $\C$ to update \eqref{RBM_IPS} in a single time step.
Since one requires the batch $\C$ to capture the binary interactions in the IPS, the least choice of the batch size is $p=2$.

The goal of this paper is to answer:
does the RB--IPS \eqref{RBM_IPS} produce accurate samples of the invariant distribution $\mu$?
Specifically, our question is two-fold:
\begin{enumerate}
\item Does the RB--IPS \eqref{RBM_IPS} has an invariant distribution $\tilde\mu$ in $\mathbb R^{Nd}$?
\item If so, what is the difference between the invariant distributions $\mu$ and $\tilde\mu$?
\end{enumerate} 

In general, the analysis of invariant distributions (which is for the long-time behavior) of the stochastic process, is more challenging than the analysis of strong and weak error in finite time. The strong and weak error analysis for the RB--IPS \eqref{RBM_IPS} has been systematically studied in \cite{RBM_error}, while the theoretical understanding of the invariant distribution is very limited, except in a random batch consensus model\cite{ha2021convergence}. 
Intuitively, we expect the trajectory $\tilde X_t$ generated by the RB--IPS \eqref{RBM_IPS}
is a good approximation to $X_t$ generated by the IPS \eqref{IPS}, since the RB--IPS provides an unbiased approximation of the interaction forces:
\begin{equation}
	\E\bigg(
		\frac1{p-1}\sum_{j\neq i,j\in\C}
		K(x^i-x^j)
	\bigg) = 
	\frac1{N-1}
	\sum_{j\neq i}
	K(x^i-x^j),~~~~
	\forall x\in\mathbb R^{Nd},
	\label{unbiased_condition}
\end{equation}
where $i$ is a fixed index in $\{1,\cdots,N\}$, and the remaining $(p-1)$ elements of the batch $\C$ are randomly chosen from $\{1,\cdots,N\}\backslash \{i\}$.
The unbiased feature \eqref{unbiased_condition} of the RB--IPS is very similar to the SGD and the SGLD.
Unfortunately, \eqref{unbiased_condition} is not sufficient to give the long-time behavior of the RB--IPS. Essentially, we lack the knowledge of the ergodicty of the RB--IPS.

The geometric ergodicity of a general stochastic process depicts how fast the distribution law converges to the invariant distribution. For the overdamped and the underdamped Langevin dynamics, the classical approaches to derive geometric ergodicity include the hypocoercivity method\cite{hypocoercivity_1,lelievre2016,hypocoercivity_3}, functional inequalities \cite{fi_1,fi_2} and the Harris ergodic theorem\cite{framework_2,Harris_1,Harris_2,Harris_3,Harris_4}. However, it is not clear how these approaches could be applied to the RB--IPS \eqref{RBM_IPS}. The main difficulty of the RB--IPS is that, the structure of the SDE varies in different time steps, preventing direct analysis of the generator.

Recently, the reflection coupling\cite{reflection_original,reflection_rate} has been employed to prove the geometric ergodicity of the overdamped Langevin dynamics, which is rather different from the classical PDE approaches. The basic idea of reflection coupling is to couple the Wiener processes of two dynamics $X_t,Y_t$ in a specially designed regime, and prove the distance  $\E[\rho(X_t,Y_t)]$ decays exponentially in time.
In particular, the reflection coupling does not require the strong convexity of the potential function $U(x)$ in \eqref{invariant_mu}. 
So far, the reflection coupling has been employed to prove the geometric ergodicity of a large variety of dynamical systems: second-order Langevin dynamics\cite{reflection_Langevin}, Hamiltonian Monte Carlo\cite{reflection_HMC_1,reflection_HMC_2}, the Andersen dynamics\cite{reflection_Andersen} and the McKean-Vlasov process\cite{reflection_mean_field}. In particular, it has been proved in \cite{reflection_rate} that the convergence rate of the IPS \eqref{IPS} does not depend on the number of particles $N$.

Using the geometric ergodicity together with the Banach fixed point theorem yields the existence of the invariant distribution $\tilde\mu$ of the RB--IPS, thus answers the first question. For the second question, we shall employ the general framework described below to estimate the difference between $\mu$ and $\tilde\mu$.
Denote the transition kernels of the IPS \eqref{IPS} and the RB--IPS \eqref{RBM_IPS} by $p_t$ and $\tilde p_t$, respectively.
After choosing a distance function $d(\cdot,\cdot)$ of probability distributions, the estimate of $d(\mu,\tilde\mu)$ relies on two key conclusions:
\begin{enumerate}
\item \textbf{Geometric ergodicity.} For the transition kernels of the IPS and the RB--IPS, there exists a constant $c>0$ such that
\begin{equation}
	d(\mu p_t,\nu p_t) \Le e^{-ct} d(\mu,\nu),~~~
	d(\mu \tilde p_t,\nu \tilde p_t) \Le e^{-ct} d(\mu,\nu),~~~
	\forall t\Ge0
\end{equation}
for any probability distributions $\mu,\nu$ in $\mathbb R^{Nd}$. The geometric ergodicity can be derived using reflection coupling.
\item \textbf{Finite-time error estimation.} Roughly speaking, we aim to prove
\begin{equation}
	\sup_{0\Le t\Le T}
	d(\nu p_t,\nu \tilde p_t) \Le C(T) \tau^\alpha
	\label{finite_time_error}
\end{equation}
for given initial distribution $\nu$ and some exponent $\alpha>0$, where the constant $C(T)$ depends on simulation time $T$. The strong error estimation derived in \cite{RBM_error} implies (\ref{finite_time_error}) with $d(\cdot,\cdot)$ being the Wasserstein distance and $\alpha = 1/2$.
\end{enumerate}
Using these conclusions, $d(\mu,\tilde\mu)$ can be estimated as follows. For any $t\Ge0$, one has the triangle inequality
\begin{align}
	d(\mu,\tilde\mu) & =
	d(\mu p_t,\tilde\mu\tilde p_t) \notag \\
	& \Le d(\mu p_t,\mu\tilde p_t) + d(\mu\tilde p_t,\tilde\mu \tilde p_t) \notag \\
	& \Le d(\mu p_t,\mu\tilde p_t) + e^{-ct}d(\mu,\tilde\mu).
	\label{triangle_inequality}
\end{align}
By choosing $t$ satisfying $e^{-ct}=1/2$, one obtains
\begin{equation}
	d(\mu,\tilde\mu) \Le C \cdot d(\mu p_t,\mu \tilde p_t) \Le C(t) \tau^\alpha.
	\label{d_order_alpha}
\end{equation}
The triangle inequality \eqref{triangle_inequality}, inspired from \cite{framework_1,framework_2,framework_3}, is the key step in the framework of estimating $d(\mu,\tilde\mu)$. The logic behind this framework is simple: geometric ergodicity and finite time error estimation imply error in invariant distributions.

Our main result in in this paper is briefly described below. Under appropriate conditions on the drift force $b(\cdot)$ and the interaction force $K(\cdot)$, the RB--IPS has geometric ergodicity and the convergence rate does not depend on the number of particles $N$, the time step $\tau$ or the batch size $p$. Also, the Wasserstein distance between $\mu,\tilde\mu$ is estimated as
\begin{equation}
	\W_1(\mu,\tilde\mu) \Le C\tau^{\frac12},
\end{equation}
where the constant $C$ does not depend on $N,\tau,p$, and the Wasserstein distance $\W_1(\cdot,\cdot)$ is defined in \eqref{W_1_distance}.
We would like to point out that our result shows that the RBM, even as an approximation to the invariant measure---which corresponds to the steady state of the system---has a convergence rate \emph{independent of} $N$.

The paper is organized as follows. Section \ref{section:ergodicity} proves the geometric ergodicity of both the IPS \eqref{IPS} and the RB--IPS \eqref{RBM_IPS}. Section \ref{section:error} proves of existence of invariant distributions and the strong error estimation in finite time, then estimates the difference between the invariant distributions of the IPS and the RB--IPS.
\section{Geometric Ergodicity of RB--IPS}
\label{section:ergodicity}

In this section we prove the geometric ergodicity of the RB--IPS \eqref{RBM_IPS}, and the main technique is the reflection coupling\cite{reflection_original,reflection_rate}. Following the methodology in \cite{reflection_rate}, we first study the geometric ergodicity of a general multiparticle system: the product model, then apply the results to the IPS \eqref{IPS} and the RB--IPS \eqref{RBM_IPS}.

The product model refers to the stochastic process of the particle system $X_t = \{X_t^i\}_{i=1}^N$ in $\mathbb R^{Nd}$, which is given by the SDE
\begin{equation}
	\d X_t^i = b^i(X_t)\d t + \sigma\d W_t,~~~~
	i=1,\cdots,N.
	\label{product_model}
\end{equation}
where $b^i(\cdot):\mathbb R^{Nd}\rightarrow\mathbb R^d$ is the total force exerted on the $i$-th particle.
The product model is so named because it is defined on the product space $\mathbb R^{Nd} = \otimes_{i=1}^N \mathbb R^d$.
Assume $b^i(x)$ is given by
\begin{equation}
	b^i(x) = b(x^i) + \gamma^i(x),~~~~
	i=1,\cdots,N,
\end{equation}
where $\gamma^i(\cdot):\mathbb R^{Nd}\rightarrow\mathbb R^d$ is the perturbation on the common drift force $b(\cdot):\mathbb R^d\rightarrow\mathbb R^d$ exerted on each particle.
Formally, the IPS \eqref{IPS} and the RB--IPS \eqref{RBM_IPS} can be unified in the product model \eqref{product_model}. In fact, the product model directly becomes the IPS by choosing
\begin{equation}
	\gamma^i(x) = \frac1{N-1}
	\sum_{j\neq i} K(x^i-x^j),~~~~
	i=1,\cdots,N.
	\label{gamma_IPS}
\end{equation}
Within each time interval $[t_n,t_{n+1})$, the RB--IPS can be viewed as the product model with
\begin{equation}
	\gamma^i(x) = \frac1{p-1}
		\sum_{j\neq i,j\in\C} K(x^i-x^j),~~~~
		i\in\C,
	\label{gamma_RBM}
\end{equation}
where $\C$ is the batch which contains $i$. Note that $\gamma^i(x)$ in the RB-IPS varies in every time step due to the use of random batches, but we have suppressed the appearance of such dependence for simplicity.

In the following, we shall use the notation $X_t = \{X_t^i\}_{i=1}^N$ to represent both the IPS \eqref{IPS} and the product model \eqref{product_model}, and the notation $\tilde X_t = \{\tilde X_t^i\}_{i=1}^N$ to represent the RB--IPS \eqref{RBM_IPS}. Using the same notation for the IPS and the product model will not be ambiguous since the the two dynamics are directly related by \eqref{gamma_IPS}.

\subsection{Product model}

We prove the geometric ergodicity of the product model \eqref{product_model}. 
Basically, we shall show that the transition kernel $p_t$ of the product model is contractive, i.e., for some $c>0$ it holds that
\begin{equation}
	d(\mu p_t,\nu p_t) \Le e^{-ct} d(\mu,\nu)
	\label{d_contractivity}
\end{equation}
for any probability distributions $\mu,\nu$ in $\mathbb R^{Nd}$.
The constant $c$ is also referred to as the \emph{contraction rate} of the dynamics.
The contractivity \eqref{d_contractivity} can be achieved by considering a coupled dynamics $\{(X_t,Y_t)\}_{t\Ge0}$ in $\mathbb R^{Nd}\times\mathbb R^{Nd}$, which is described as:
\begin{enumerate}
\item The initial values $X_0\sim\mu$ and $Y_0\sim\nu$ (not necessarily independent);
\item Both $\{X_t\}_{t\Ge0}$ and $\{Y_t\}_{t\Ge0}$ are weak solutions to the product model \eqref{product_model};
\item $X_t,Y_t$ are driven by two Wiener processes $W_t^X,W_t^Y$ respectively, while $W_t^X,W_t^Y$ are coupled in a specific regime.
\end{enumerate}
The coupled dynamics $\{(X_t,Y_t)\}_{t\Ge0}$ can also be written as the SDE
\begin{equation}
\left\{
\begin{aligned}
	\d X_t^i & = b^i(X_t)\d t + \sigma\d W_t^{X,i} \\
	\d Y_t^i & = b^i(Y_t)\d t + \sigma\d W_t^{Y,i}
\end{aligned}
\right.~~~~
i=1,\cdots,N,
\label{coupled_dynamics}
\end{equation}
where $W_t^{X,i},W_t^{Y,i}$ are the $i$-th arguments of the Wiener processes $W_t^X,W_t^Y$ in $\mathbb R^{Nd}$. If one proves for some distance function $\rho(\cdot,\cdot)$ in $\mathbb R^{Nd}\times\mathbb R^{Nd}$,
the expectation $\E[\rho(X_t,Y_t)]$ has exponential decay in time, i.e., for some $c>0$ it holds that
\begin{equation}
	\E[\rho(X_t,Y_t)] \Le 
	e^{-ct} \E[\rho(X_0,Y_0)],
	\label{rho_contractivity}
\end{equation}
then the contractivity \eqref{d_contractivity} holds with $d(\cdot,\cdot)$ being the Wasserstein distance
\begin{equation}
	d(\mu,\nu) := \inf_{\gamma\in\Pi(\mu,\nu)}
	\int_{\mathbb R^{Nd}\times\mathbb R^{Nd}}
	\rho(x,y) \gamma(\d x\d y),
\end{equation}
where $\Pi(\mu,\nu)$ is the set of joint distributions in $\mathbb R^{Nd}\times\mathbb R^{Nd}$ with marginal distributions $\mu,\nu$. The concept of Wasserstein distance has been widely adopted in optimal transport\cite{optimal_transport_1,optimal_transport_2}, where $\Pi(\mu,\nu)$ is known as the set of transport plans.

In the definition of the coupled dynamics $\{(X_t,Y_t)\}_{t\Ge0}$, we expect the coupling scheme between the Wiener processes $W_t^X,W_t^Y$ attracts $X_t,Y_t$ together so that the estimate \eqref{rho_contractivity} holds. Note that the coupling scheme between $W_t^X,W_t^Y$ deos not impact the fact that $X_t\sim \mu p_t$ and $Y_t\sim \nu p_t$, as long as one fixes the initial distributions $\mu,\nu$. In other words, the choice of the coupling scheme between $W_t^X,W_t^Y$ is flexible in the proof of contractivity \eqref{d_contractivity}. Therefore, our goal is to find an appropriate coupling scheme between $W_t^X,W_t^Y$ so that \eqref{rho_contractivity} holds.

The simplest coupling scheme is $W_t^X = W_t^Y$, which is also known as the synchronous coupling\cite{reflection_original}. The synchronous coupling can be used to prove the contractivity \eqref{rho_contractivity} when $b(x) = -\nabla U(x)$ and the potential function $U(x)$ is strongly convex. Unfortunately, the synchronous coupling cannot directly apply to the general case when $U(x)$ is not convex.

Another choice is the reflection coupling.
In \cite{reflection_original}, the reflection coupling is used to prove the contractivity of the overdamped Langevin dynamics of a single particle. Later in \cite{reflection_rate}, this approach is used to prove the contractivity of the product model \eqref{product_model}. In this paper we shall review the reflection coupling for the product model and generalize the results to the IPS and the RB--IPS.

Consider the coupling scheme for the product model \eqref{product_model} introduced in \cite{reflection_rate}.
For this $N$-particle system, each pair of particles $\{(X_t^i,Y_t^i)\}_{t\Ge0}$ is evolved by
\begin{equation}
\left\{
\begin{aligned}
	\d X_t^i & = b^i(X_t)\d t + \sigma\lambda(Z_t^i)\d W_t^i + \sigma\pi(Z_t^i)\d \tilde W_t^i \\
	\d Y_t^i & = b^i(Y_t)\d t + \sigma\lambda(Z_t^i)(I-2e_t^i(e_t^i)^\mathrm{T})\d W_t^i 
	+ \sigma\pi(Z_t^i)\d \tilde W_t^i
\end{aligned}
\right.~~~~
i=1,\cdots,N,
\label{product_couple}
\end{equation}
where $Z_t^i = X_t^i - Y_t^i$, $e_t^i = Z_t^i/|Z_t^i|$, and $\{W_t^i\}_{i=1}^N,\{\tilde W_t^i\}_{i=1}^N$ are independent Wiener processes in $\mathbb R^d$.
Besides, $\lambda(z),\pi(z)$ are smooth functions satisfying
\begin{equation}
	\lambda^2(z) + \pi^2(z) = 1,~~~~
	\forall z\in\mathbb R^d
	\label{lambda_pi}
\end{equation}
with $\lambda(z) = 0$ for $|z|\Le\delta/2$ and $\lambda(z) = 1$ for $|z|\Ge \delta$. Clearly, for each $i\in\{1,\cdots,N\}$, the dynamics $X_t^i,Y_t^i$ in $\mathbb R^{d}$ are driven by the stochastic processes
\begin{align}
	W_t^{X,i} & = \int_0^t \big(
	\lambda(Z_s^i)\d W_s^i + 
	\pi(Z_s^i)\d \tilde W_s^i
	\big), \\
	W_t^{Y,i} & = 
	\int_0^t \big(
	\lambda(Z_s^i) (I-2e_s^i(e_s^i)^\mathrm{T}) \d W_s^i + 
	\pi(Z_s^i)\d \tilde W_s^i
	\big),
\end{align}
respectively.
We present some intuitive explanations of the coupled dynamics \eqref{product_couple}:
\begin{enumerate}
\item The coupled dynamics \eqref{product_couple} is a mixture of the synchronous coupling ($\d \tilde W_t^i$) and the reflection coupling ($\d W_t^i$). The matrix $I-2e_t^i(e_t^i)^\mathrm{T}\in\mathbb R^{d\times d}$ is the reflection transform with respect to the normal plane of $e_t^i$, which is the reason the $(\d W_t^i)$ part is called \emph{reflection coupling}.
\item By Levy's characterization\cite{levy}, the normalizing condition \eqref{lambda_pi} ensures that both $W_t^{X,i},W_t^{Y,i}$ are standard Wiener processes in $\mathbb R^d$. Therefore, both dynamics $X_t,Y_t$ are weak solutions to the product model \eqref{product_model}.
\item $\delta>0$ is a free parameter in the definition of the coupled dynamics \eqref{product_couple}.
Since $Z_t^i = X_t^i - Y_t^i$ is the relative displacement between $X_t^i,Y_t^i$, we have:
\begin{itemize}
\item When $|Z_t^i| \Ge \delta$, $\lambda(Z_t^i)\equiv1$ and \eqref{product_couple} is fully reflection coupling. 
\item When $|Z_t^i| \Le \delta/2$, \eqref{product_couple} degenerates to fully synchronous coupling.
\end{itemize}
When $\delta$ is sufficiently small, we expect that $\lambda(z)$ is close to the constant function 1 and thus the reflection coupling dominates the coupled dynamics \eqref{product_couple}.
\end{enumerate}
\begin{remark}
In \cite{reflection_original}, the coupling scheme for a single particle is fully reflection coupling, i.e., $\lambda(z) \equiv 1$.
However, if we simply choose $\lambda(z) \equiv 1$ in the product model \eqref{product_model}, it is inconvenient to define the coupled dynamics after the occurrence of $Z_t^i = 0$. Also as indicated in \cite{reflection_rate},
it is difficult to make the proof of contractivity rigorous when $\lambda(z)\equiv1$.
\end{remark}

From the coupled dynamics \eqref{product_couple}, the displacement $Z_t^i$ satisfies the SDE
\begin{equation}
	\d Z_t^i = (b^i(X_t)-b^i(Y_t))\d t + 
	2\sigma\lambda(Z_t^i)|Z_t^i|^{-1}Z_t^i\d B_t^i,
	\label{SDE_of_Z}
\end{equation}
where $B_t^i$ is the 1D Wiener process defined by
\begin{equation}
	B_t^i = \int_0^t (e_s^i)^\mathrm{T} \d W_s^i.
\end{equation}
Note that the synchronous coupling ($\d \tilde W_t^i$) vanishes in \eqref{SDE_of_Z}, and the diffusion coefficient $\sigma \lambda(Z_t^i)$ comes from reflection coupling ($\d W_t^i$). Since we expect the diffusion term attracts $X_t^i,Y_t^i$ together, the condition $\sigma>0$ is essential in the proof of contractivity.
Let $r_t^i = |Z_t^i|$, then $r_t^i$ satisfies
\begin{equation}
	\d r_t^i =  
	(r_t^i)^{-1}
	Z_t^i\cdot (b^i(X_t) - b^i(Y_t))) \d t + 
	2\sigma\lambda (Z_t^i)\d B_t^i.
	\label{SDE_of_r}
\end{equation}

Choosing a distance function $f(r)\in C^2[0,+\infty)$, by It\^o's formula, one obtains
\begin{equation}
	\d f(r_t^i) = 
	2\sigma \lambda(Z_t^i)
	f'(r_t^i)\d B_t^i
	+ \Big((r_t^i)^{-1}
	Z_t^i\cdot (b^i(X_t) - b^i(Y_t))) f'(r_t^i) + 
	2\sigma^2\lambda^2(Z_t^i) f''(r_t^i)\Big)\d t,
\end{equation}
hence the rate of change for $\E[f(r_t^i)]$ is
\begin{equation}
	\frac{\d}{\d t} \E[f(r_t^i)] = 
	\E\Big(
		(r_t^i)^{-1} Z_t^i \cdot 
		(b^i(X_t) - b^i(Y_t)) f'(r_t^i)
		+ 2\sigma^2 \lambda^2(Z_t^i) f''(r_t^i)
	\Big).
	\label{rate_of_change_f}
\end{equation}
Now define the distance $\rho(\cdot,\cdot)$ between the systems $X_t,Y_t\in\mathbb R^{Nd}$ by
\begin{equation}
	\rho(X_t,Y_t) := 
	\frac1N\sum_{i=1}^N f(r_t^i),
	\label{rho}
\end{equation}
then the rate of change for $\E[\rho(X_t,Y_t)]$ is completely given by \eqref{rate_of_change_f}\eqref{rho}.

In order to prove $\E[\rho(X_t,Y_t)]$ has exponential decay in time as in \eqref{rho_contractivity}, we put some technical assumptions on the drift forces $\{b^i(x)\}_{i=1}^N$. The distance function $f(r)$ will also be chosen according to these assumptions. Since each $b^i(x) = b(x^i) + \gamma^i(x)$, we only need to consider the assumptions on the drift force $b(\cdot)$ and the perturbation $\gamma^i(\cdot)$.

For the drift force $b(\cdot):\mathbb R^d\rightarrow\mathbb R^d$, suppose there is a function $\kappa(r)$ satisfying
\begin{equation}
	\kappa(r) \Le 
	\inf
	\bigg\{
	-\frac{2}{\sigma^2}
	\frac{(x-y)\cdot (b(x) - b(y))}{|x-y|^2}:
	x,y\in\mathbb R^d,
	|x-y|=r
	\bigg\}.
	\label{kappa}
\end{equation}
Roughly speaking, when $b(x) = -\nabla U(x)$, the function $\kappa(r)$ depicts the convexity of the potential function $U(x)$. If the Hessian $\nabla^2 U(x)$ stays positive definite outside a finite spherical region, then $\kappa(r)$ is positive for sufficiently large $r$. Therefore, it is reasonable to require the positivity of $\kappa(r)$.
\begin{assumption}[drift]
\label{assumption:kappa}
The function $\kappa(r)$ defined in \eqref{kappa} satisfies
\begin{enumerate}
\item[\textnormal{1.}] $\kappa(r)$ is continuous for $r\in(0,+\infty)$;
\item[\textnormal{2.}] $\kappa(r)$ has a lower bound for $r\in(0,+\infty)$;
\item[\textnormal{3.}] $\di\varliminf_{r\rightarrow\infty} \kappa(r) >0$.
\end{enumerate}
\end{assumption}
For the perturbation $\gamma^i(\cdot):\mathbb R^{Nd}\rightarrow\mathbb R^d$, assume the Lipschitz condition holds:
\begin{assumption}[perturbation]
\label{assumption:gamma}
There exists a constant $L$ such that
\begin{equation}
	\sum_{i=1}^N 
	|\gamma^i(x) - \gamma^i(y)| \Le
	L \sum_{i=1}^N |x^i-y^i|,~~~~
	\forall x,y\in\mathbb R^{Nd}.
\end{equation}
\end{assumption}
\begin{remark}
If a non-continuous function $\kappa(r)$ satisfies \eqref{kappa} and the latter two conditions in Assumption \textnormal{\ref{assumption:kappa}}, we can find another continious function $\bar\kappa(r)\Le \kappa(r)$ which satisfies all conditions in Assumption \textnormal{\ref{assumption:kappa}}. Therefore, the continuity of $\kappa(r)$ is not an essential condition in Assumption \textnormal{\ref{assumption:kappa}}.
We assume the continuity of $\kappa(r)$ merely for technical convenience.
\end{remark}

Following \cite{reflection_original,reflection_rate}, we choose the distance function $f(r)$ according to the following lemma.
\begin{lemma}[distance]If the function $\kappa(r)$ defined in \eqref{kappa} satisifies Assumption \textnormal{\ref{assumption:kappa}}, then there exists a function $f(r)$ defined in $r\in[0,+\infty)$ such that
\begin{enumerate}
\item[$1.$] $f(0) = 0$, and $f(r)$ is concave and strictly increasing in $[0,+\infty)$;
\item[$2.$] $f(r)\in C^2[0,+\infty)$ and there exists a constant $c_0>0$ such that
\begin{equation}
	f''(r) - \frac14 r\kappa(r) f'(r) \Le - \frac{c_0}2f(r),
	~~~~\forall r\Ge 0
	\label{f_inequality}
\end{equation}
\item[$3.$] There exists a cosntant $\varphi_0>0$ such that
\begin{equation}
	\frac{\varphi_0}4r\Le f(r)\Le r,~~~~\forall r\Ge0.
	\label{f_equivalence}
\end{equation}
\end{enumerate}
The constants $c_0,\varphi_0$ only depend on the function $\kappa(r)$.
\label{lemma:distance}
\end{lemma}
The choice of distance function $f(r)$ is to produce the negative coefficient in the RHS of \eqref{f_inequality}, and the proof of Lemma \ref{lemma:distance} is left in Appendix. Fig.~\ref{figure:f_distance_function} is an example of the distance function $f(r)$ in the case $\kappa(r) = \max\{r/2\sqrt{2}-1,1\}$, where the graphs of $\kappa(r)$ and $f(r)$ are shown.
\begin{figure}
	\centerline{\includegraphics[width=0.75\textwidth]{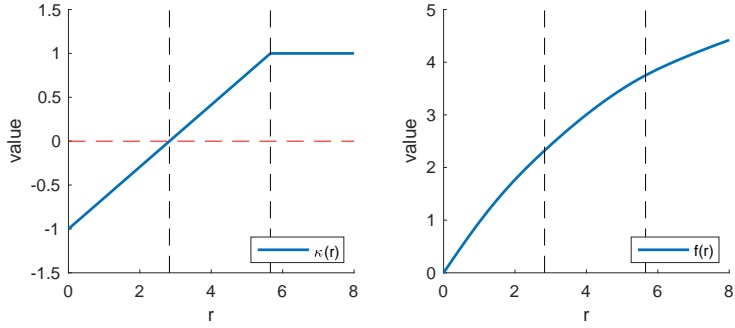}}
	\caption{Graphs of $\kappa(r)$ (left) and $f(r)$ (right), where $f(r)$ is defined according to Lemma \ref{lemma:distance}.}
	\label{figure:f_distance_function}
\end{figure}
In Fig. \ref{figure:f_distance_function}, we observe that $f(r)$ is concave for small $r$ and almost linear for large $r$. Here is an intuitive explanation how this feature of $f(r)$ is related to the inequality \eqref{f_inequality}. For simplicity, assume the drift force $b(x) = -\nabla U(x)$.
When $r_t^i = |X_t^i-Y_t^i|$ is large, the particles $X_t^i,Y_t^i$ are attracted together by their common drift force. When $r_t^i$ is small, the concavity of $f(r)$ makes the quantity $f(r_t^i)$ more sensitive to the decreasing rather than increasing of the relative distance $r_t^i$, and thus we can expect the decreasing of $\E[f(r_t^i)]$ even without the global convexity of the potential function $U(x)$.

Using the distance function $f(r)$ defined in Lemma \ref{lemma:distance}, we are able to estimate the rate of change for $\E[\rho(X_t,Y_t)]$. The following lemma is a key step to derive the estimation:
\begin{lemma}[$f$-inequality]
\label{lemma:sum_f_inequality}
Under Assumptions \textnormal{\ref{assumption:kappa}} and \textnormal{\ref{assumption:gamma}}, let $f(r)$ be the distance function given in Lemma \textnormal{\ref{lemma:distance}}.
Given $\delta>0$, let $\lambda(z)$ be a smooth continuous function with $|\lambda(z)|\Le 1$ and $\lambda(z)=1$ for $|z|\Ge \delta$. If the Lipschitz constant $L$ in Assumption \textnormal{\ref{assumption:gamma}} satisifies
$$
	L < \frac{c_0\varphi_0\sigma^2}{8},
$$
then the following inequality holds with $c := c_0\sigma^2/2$,
\begin{equation}
	\sum_{i=1}^N \bigg(
	(r^i)^{-1}Z^i \cdot (b^i(X) - b^i(Y))f'(r^i) + 
	2\sigma^2 \lambda^2(Z^i) f''(r^i) 
	\bigg)
	\Le 
	Nm(\delta) - c\sum_{i=1}^N f(r^i),
	\label{sum_f_inequality}
\end{equation}
where $X,Y\in\mathbb R^{Nd}$, $Z = X-Y$, $r^i = |Z^i|$ and $m(\delta)$ is defind by
\begin{equation}
	m(\delta) = \frac{\sigma^2}{2} \sup_{r<\delta}
		\Big(r\kappa(r)^-\Big) 
		+ c_0 \sigma^2 \delta.
	\label{m_delta}
\end{equation}
$x^- = -\min\{x,0\}$ denotes the negative part of $x\in\mathbb R$.
\end{lemma}
The proof of Lemma \ref{lemma:sum_f_inequality} is left in Appendix, and is similar to the proof of Theorem 7 in \cite{reflection_rate}.
Also note that \eqref{sum_f_inequality} is exactly the condition of Lemma 5 in \cite{reflection_rate}. We present some comments on Lemma \ref{lemma:sum_f_inequality}:
\begin{enumerate}
\item $\delta$ and $\lambda(z)$ in Lemma \ref{lemma:sum_f_inequality} correspond to parameters in the coupled dynamics \eqref{product_couple}. The additional term $m(\delta)$ appears in the RHS of \eqref{sum_f_inequality} because $\lambda(z)$ is not identical to 1, i.e., we are not using fully reflection coupling. By Assumption \ref{assumption:kappa}, $\kappa(r)^-$ is bounded for $r\in(0,+\infty)$, thus one has $\di\lim_{\delta\rightarrow0} m(\delta) = 0$.
\item The validity of \eqref{sum_f_inequality} requires the Lipschitz constant $L$ to be sufficiently small, and the $\sigma^2$ term appears in the upper bound of $L$. In other words, the strength of the diffusion needs to strong enough to control the error due to the perturbation of drift forces. In fact, if $\sigma=0$, the product model \eqref{product_model} as an overdamped Langevin dynamics will degenerate to the gradient flow, and may converge to the local minimizers of the potential function. In this case we cannot expect that \eqref{product_model} has a unique invariant distribution or the contraction property.
\item The distance function $f(r)$, the upper bound of $L$ and the contraction rate $c$ are all independent of $\delta$, thus we may pass $\delta$ to the limit 0 without changing the value of $c$.
\end{enumerate}
Using Lemma \ref{lemma:sum_f_inequality}, we can obtain the contractivity of the coupled dynamics \eqref{product_couple}.
\begin{lemma}[contractivity]
\label{lemma:rho_rate}
Under Assumptions \textnormal{\ref{assumption:kappa}} and \textnormal{\ref{assumption:gamma}}, let $f(r)$ be the distance function defined in Lemma \textnormal{\ref{lemma:distance}}, and $c:=c_0\sigma^2/2$. If the Lipschitz constant $L$ in Assumption \textnormal{\ref{assumption:gamma}} satisfies
$$
	L<\frac{c_0\varphi_0\sigma^2}{8},
$$
then for $\rho_t:=\rho(X_t,Y_t)$ defined in \eqref{rho}, one has
\begin{equation}
	\frac{\d}{\d t}
	\E[\rho_t] \Le m(\delta) - c\cdot \E[\rho_t],~~~~\forall t\Ge0,
	\label{rate_of_change_rho}
\end{equation}
where $m(\delta)$ is defined in \eqref{m_delta}.
\end{lemma}
\begin{proof}
Since $\rho_t = \sum_{i=1}^N f(r_t^i)/N$, one has
\begin{equation}
	\frac{\d}{\d t}\E[\rho_t] = 
	\frac1N \sum_{i=1}^N \frac{\d}{\d t} \E[f(r_t^i)]
\end{equation}
Using \eqref{rate_of_change_f}, one obtains
\begin{equation}
	\frac{\d}{\d t}\E[\rho_t] = 
	\frac1N\sum_{i=1}^N
	\E\Big(
		(r_t^i)^{-1} Z_t^i \cdot 
		(b^i(X_t) - b^i(Y_t)) f'(r_t^i)
		+ 2\sigma^2 \lambda^2(Z_t^i) f''(r_t^i)
	\Big) 
\end{equation}
Applying the estimate in Lemma \ref{lemma:sum_f_inequality}, one obtains
\begin{equation}
	\frac{\d}{\d t}\E[\rho_t] \Le 
	m(\delta) - c\cdot\E[\rho_t],
\end{equation}
which is exactly the desired result.
\end{proof}
Integrating \eqref{rate_of_change_rho} in the time interval $[0,t)$ gives
\begin{equation}
	\E[\rho_t] \Le 
	e^{-ct}\E[\rho_0] + 
	\frac{m(\delta)(1-e^{-ct})}c,~~~~
	\forall t\Ge0,
	\label{rho_inequality}
\end{equation}
which can be used to derive the contractivity for the probability distributions.

To describe the probability distributions rigorously, introduce the following terminologies.
Let $\P_1$ be the set of probability distributions in $\mathbb R^{Nd}$ with finite first-order moment, i.e.,
\begin{equation}
	\P_1 = \bigg\{
	\mbox{$\mu$ is a probability distribution in $\mathbb R^{Nd}$}:
	\sum_{i=1}^N\int_{\mathbb R^{Nd}}
	|x^i|\mu(\d x) < +\infty
	\bigg\}.
	\label{distribution_P1}
\end{equation}
For probability distributions $\mu,\nu\in \P_1$, define the normalized Wasserstein distances
\begin{align}
	\W_1(\mu,\nu) & = \inf_{\gamma\in\Pi(\mu,\nu)}
	\int_{\mathbb R^{Nd}\times\mathbb R^{Nd}} \bigg(
	\frac1N\sum_{i=1}^N
	|x^i-y^i|\bigg)\gamma(\d x\d y), \\
	\W_f(\mu,\nu) & = \inf_{\gamma\in\Pi(\mu,\nu)}
		\int_{\mathbb R^{Nd}\times\mathbb R^{Nd}} \bigg(
		\frac1N\sum_{i=1}^N
		f(|x^i-y^i|)\bigg)\gamma(\d x\d y).
\end{align}
It is easy to verify $(\P_1,\W_1(\cdot,\cdot))$ is a complete metric space. Note that $f(r)$ does not satisfy the triangle inquality due to concavity, $\W_f$ is only a semimetric. Since $f(r)$ is equivalent to the Euclidean norm, $\W_1$ and $\W_f$ are equivalent as semimetrics.  Using the estimate \eqref{rho_inequality}, we obtain
\begin{theorem}[contractivity]
\label{theorem:product_contractivity}
Under Assumptions \textnormal{\ref{assumption:kappa}} and \textnormal{\ref{assumption:gamma}}, let $f(r)$ be the distance function defined in Lemma \textnormal{\ref{lemma:distance}}, and $c:=c_0\sigma^2/2$.
Let $p_t$ be the transition kernel of the product model \eqref{product_model}.
If the Lipschitz constant $L$ in Assumption \textnormal{\ref{assumption:gamma}} satisfies
$$
	L<\frac{c_0\varphi_0\sigma^2}{8},
$$
then we have
\begin{equation}
	\W_f(\mu p_t,\nu p_t) \Le e^{-ct} \W_f(\mu,\nu),
	~~~~\forall t\Ge0.
	\label{W_f_contractivity}
\end{equation}
for any probability distributions $\mu,\nu\in\P_1$.
\end{theorem}
The proof of Theorem \ref{theorem:product_contractivity} is similar to the proof of Theorem 7 in \cite{reflection_rate}.
\begin{proof}
For given distributions $\mu,\nu\in \P_1$, let $\gamma\in\Pi(\mu,\nu)$ satisfies
\begin{equation}
	\int_{\mathbb R^{Nd}} \bigg(
	\frac1N\sum_{i=1}^N
	f(|x^i-y^i|)\bigg)\gamma(\d x\d y) \Le \W_f(\mu,\nu) + \varepsilon,
\end{equation}
where $\varepsilon>0$ is an arbitrary small constant. Let $\{(X_t,Y_t)\}_{t\Ge0}$ evolved by the coupled dynamics \eqref{product_couple} with the initial value $(X_0,Y_0)\sim\gamma$, then $X_t\sim \mu p_t$ and $Y_t\sim \nu p_t$. From the inequality \eqref{rho_inequality} we obtain
\begin{equation}
	\E[\rho(X_t,Y_t)] \Le 
	e^{-ct} \E[\rho(X_0,Y_0)] + \frac{m(\delta)(1-e^{-ct})}{c} \Le 
	e^{-ct} \W_f(\mu,\nu) +
	\frac{m(\delta)(1-e^{-ct})}{c} + \varepsilon
\end{equation}
Using the definition of $\W_f$, 
\begin{equation}
	\E[\rho(X_t,Y_t)]  
	\Ge \inf_{\gamma\in\Pi(\mu p_t,\nu p_t)}
	\int \bigg(
	\frac1N\sum_{i=1}^N
	f(|x^i-y^i|)\bigg)\gamma(\d x\d y) =
	\W_f(\mu p_t,\nu p_t),
\end{equation}
hence one obtains
\begin{equation}
	\W_f(\mu p_t,\nu p_t) \Le 
	e^{-ct} \W_f(\mu,\nu) +
	\frac{m(\delta)(1-e^{-ct})}{c} + \varepsilon.
\end{equation}
Note that the evolution of $\mu p_t$ and $\nu p_t$ does not depend on the coupling scheme, we can directly pass $\delta$ and $\varepsilon$ to 0 and obtain
\begin{equation}
	\W_f(\mu p_t,\nu p_t) \Le e^{-ct} \W_f(\mu,\nu),
\end{equation}
which is exactly the contractivity we need.
\end{proof}
\subsection{Exact dynamics: IPS}
We apply Theorem \ref{theorem:product_contractivity} to derive the geometric ergodicity for the IPS \eqref{IPS}. For the IPS, the perturbation $\gamma^i(\cdot):\mathbb R^{Nd}\rightarrow\mathbb R^d$ is given by \eqref{gamma_IPS}. Suppose $L_K$ is the Lipschitz constant of the interaction $K(\cdot):\mathbb R^d\rightarrow\mathbb R^d$, then for any $x,y\in\mathbb R^{Nd}$,
\begin{align*}
	|\gamma^i(x) - \gamma^i(y)| & \Le 
	\frac1{N-1}
	\sum_{j\neq i} |K(x^i - x^j) - K(y^i - y^j)| \\
	& \Le \frac{L_K}{N-1} \sum_{j\neq i}
	\big(|x^i-y^i| + |x^j - y^j|\big).
\end{align*}
Summation over $i\in\{1,\cdots,N\}$ gives
\begin{equation}
	\sum_{i=1}^N |\gamma^i(x) - \gamma^i(y)| \Le 
	2L_K \sum_{i=1}^N |x^i-y^i|.
\end{equation}
Hence Assumption \ref{assumption:gamma} holds with the constant $L=2L_K$. 
In terms of the interaction force $K(\cdot)$, we may replace Assumption \ref{assumption:gamma} by the following one:
\begin{assumption}[interaction]
\label{assumption:interact}
There exists a constant $L_K$ such that
\begin{equation}
	\max\{
	|K(x)|,|\nabla K(x)|,|\nabla^2K(x)|\}
	\Le L_K ,~~~~
	\forall x\in\mathbb R^d.
\end{equation}
\end{assumption}
\begin{remark}
Assumption \ref{assumption:interact} is stronger than Assumption \ref{assumption:gamma} because we require not only $\nabla K(\cdot)$ but also $K(\cdot)$ and $\nabla^2 K(\cdot)$ to be uniformly bounded. The boundedness of $K(\cdot)$ and $\nabla^2 K(\cdot)$ is not necessary to prove the geometric ergodicity, but will be useful in the strong error estimation in Section \ref{section:error}.
\end{remark}
For completeness, we explicitly write the coupling scheme for the IPS \eqref{IPS}. The coupled dynamics $\{(X_t,Y_t)\}_{t\Ge0}$ in $\mathbb R^{Nd}\times\mathbb R^{Nd}$ is given by
\begin{equation}
\left\{
\begin{aligned}
	\d X_t^i & = b(X_t^i)\d t + \frac1{N-1}
	\sum_{j\neq i} K(X_t^i-X_t^j)\d t
	 + \sigma\lambda(Z_t^i)\d W_t^i + \sigma\pi(Z_t^i)\d \tilde W_t^i \\
	\d Y_t^i & = b^i(Y_t)\d t  + \frac1{N-1}
	\sum_{j\neq i}K(Y_t^i-Y_t^j)\d t
	 + \sigma\lambda(Z_t^i)(I-2e_t^i(e_t^i)^\mathrm{T})\d W_t^i 
	+ \sigma\pi(Z_t^i)\d \tilde W_t^i
\end{aligned}
\right.
\label{IPS_couple}
\end{equation}
for $i=1,\cdots,N$.
Theorem \ref{theorem:product_contractivity} then immediately implies
\begin{theorem}[contractivity]
\label{theorem:IPS_contractivity}
Under Assumption \textnormal{\ref{assumption:kappa}} and \textnormal{\ref{assumption:interact}}, let $f(r)$ be the distance function defined in Lemma \textnormal{\ref{lemma:distance}}, and $c:=c_0\sigma^2/2$.
Let $p_t$ be the transition kernel of the IPS \eqref{IPS}.
If the constant $L_K$ in Assumption \textnormal{\ref{assumption:interact}} satisfies
$$
	L_K<\frac{c_0\varphi_0\sigma^2}{16},
$$
then we have
\begin{equation}
	\W_f(\mu p_t,\nu p_t) \Le e^{-ct} \W_f(\mu,\nu),
	~~~~\forall t\Ge0
\end{equation}
for any probability distributions $\mu,\nu\in\P_1$.
\end{theorem}
Theorem \ref{theorem:IPS_contractivity} is similar to
Corollary 9 in \cite{reflection_rate}.
An important observation from Theorem \ref{theorem:IPS_contractivity} is that both the contraction rate $c$ and the bound of $L_K$ does not depend on the number of particles $N$. A direct corollary of Theorem \ref{theorem:IPS_contractivity} is that for any initial distribution $\nu \in \P_1$, $\nu p_t$ converges to the invariant distribution $\mu$ exponentially.
\begin{corollary}[ergodicity]
\label{theorem:IPS_ergodicity}
Under Assumption \textnormal{\ref{assumption:kappa}} and \textnormal{\ref{assumption:interact}}, let $f(r)$ be the distance function defined in Lemma \textnormal{\ref{lemma:distance}}, and $c:=c_0\sigma^2/2$.
Let $p_t$ be the transition kernel of the IPS \eqref{IPS}, and $\mu\in\P_1$ be the invariant distribution. If the constant $L_K$ in Assumption \textnormal{\ref{assumption:interact}} satisfies
$$
	L_K<\frac{c_0\varphi_0\sigma^2}{16},
$$
then we have
\begin{equation}
	\W_f(\nu p_t,\mu) \Le e^{-ct} \W_f(\nu,\mu),
	~~~~\forall t\Ge0
\end{equation}
for any probability distribution $\nu\in\P_1$.
\end{corollary}
\noindent
The existence of the invariant distribution $\mu$ will be later proved in Theorem \ref{theorem:existence}.

\subsection{Random batch dynamics: RB--IPS}

We prove the geometric ergodicity of the RB--IPS \eqref{RBM_IPS} using reflection coupling. Unfortunately, Theorem \ref{theorem:IPS_contractivity} cannot be directly applied since the perturbation $\gamma^i(x)$ changes its expression in different time steps. In the following, proof of contractivity for the RB--IPS will be mainly based on Lemma \ref{lemma:rho_rate}. Also, it is necessary to clarify the coupled dynamics for the RB--IPS \eqref{RBM_IPS}.

Suppose at the time step $t_n$, the division $\D_n = \{\C_1,\cdots,\C_q\}$ is randomly generated, then the perturation $\gamma^i(x)$ within the time interval $[t_n,t_{n+1})$ is given by \eqref{gamma_RBM}. It is easy to verify
\begin{align*}
	|\gamma^i(x) - \gamma^i(y)| & \Le \frac1{p-1}
	\sum_{j\neq i,j\in\C}
	|K(x^i-x^j) - K(y^i-y^j)| \\
	& \Le \frac{L_K}{p-1}\sum_{j\neq i,j\in\C}
	(|x^i-y^i|+|x^j-y^j|)
\end{align*}
where $\C\in\D_n$ is the batch which contains $i$.
Summation over $i\in\C$ gives
\begin{equation}
	\sum_{i\in\C} |\gamma^i(x) - \gamma^i(y)|
	\Le 2L_K
	\sum_{i\in\C} 
	|x^i-y^i|
\end{equation}
Summation over $\C\in\{\C_1,\cdots,\C_q\}$ gives
\begin{equation}
	\sum_{i=1}^N |\gamma^i(x) - \gamma^i(y)| \Le 
	2L_K
	\sum_{i=1}^N |x^i-y^i|.
\end{equation}
Hence Assumption \ref{assumption:interact} still holds with $L=2L_K$. In a similar way, define the coupled dynamics for the RB--IPS \eqref{RBM_IPS} as follows.

Fix the parameter $\delta>0$ and let the smooth functions $\lambda(z),\pi(z)$ be defined as in \eqref{lambda_pi}.
At each time step $t_n$, suppose the division $\D_n$ is randomly generated, and the coupled dynamics $\{(\tilde X_t,\tilde Y_t)\}_{t\Ge0}$ in $\mathbb R^{Nd}\times\mathbb R^{Nd}$ within the time interval $[t_n,t_{n+1})$ is defined by
\begin{equation}
\left\{
\begin{aligned}
	\d \tilde X_t^i & = b(\tilde X_t^i)\d t + \frac1{p-1}
	\sum_{j\neq i,j\in\C} K(\tilde X_t^i-\tilde X_t^j)\d t
	 + \sigma\lambda(\tilde Z_t^i)\d W_t^i + \sigma\pi(\tilde Z_t^i)\d \tilde W_t^i \\
	\d \tilde Y_t^i & = b^i(\tilde Y_t)\d t  + \frac1{p-1}
	\sum_{j\neq i,j\in\C}K(\tilde Y_t^i-\tilde Y_t^j)\d t
	 + \sigma\lambda(\tilde Z_t^i)(I-2e_t^i(e_t^i)^\mathrm{T})\d W_t^i 
	+ \sigma\pi(\tilde Z_t^i)\d \tilde W_t^i
\end{aligned}
\right.
\label{RBM_couple}
\end{equation}
for $i\in\C$ and $\C\in\D_n$, where $\tilde Z_t^i = \tilde X_t^i - \tilde Y_t^i$ and $e_i = \tilde Z_t^i/|\tilde Z_t^i|$. For convenience, define the filtration of the coupled dynamics \eqref{RBM_couple} by
\begin{equation}
	\mathcal G_n = \sigma(
	(\tilde X_0,\tilde Y_0),
	\{W_s\}_{0\Le s\Le t_n},
	\{\tilde W_s\}_{0\Le s\Le t_n},
	\{\D_k\}_{0\Le k\Le n}).
\end{equation}
That is, $\mathcal G_n$ is determined by the joint distribution of $(\tilde X_0,\tilde Y_0)$ in $\mathbb R^{Nd}\times\mathbb R^{Nd}$, Wiener processes $W_t,\tilde W_t$ before $t_n$, and the batch divisions in the first $n+1$ time steps. Under the condition of $\mathcal G_n$, the division $\D_n$ within the time step $[t_n,t_{n+1})$ is determined, and the coupled dynamics of $(\tilde X_t,\tilde Y_t)$ is exactly given by \eqref{RBM_couple}.

We still choose the distance function $f(r)$ according to Lemma \ref{lemma:distance}, and the distance between $\tilde X_t,\tilde Y_t\in\mathbb R^{Nd}$ is defined by
\begin{equation}
	\rho(\tilde X_t,\tilde Y_t) = \frac1N \sum_{i=1}^N 
	f(\tilde r_t^i),
	\label{tilde_rho}
\end{equation}
where $\tilde r_t^i = |\tilde Z_t^i|$. Similar to Lemma \ref{lemma:rho_rate}, we may derive the contractivity for the coupled dynamics \eqref{RBM_couple}, but only in the time interval $[t_n,t_{n+1})$ and under the condition of fixed $\mathcal G_n$.
\begin{corollary}[contractivity]
\label{theorem:random_rho_rate}
Under Assumptions \textnormal{\ref{assumption:kappa}} and \textnormal{\ref{assumption:interact}}, let $f(r)$ be the distance function in Lemma \textnormal{\ref{lemma:distance}}, and $c:=c_0\sigma^2/2$. If the constant $L_K$ in Assumption \textnormal{\ref{assumption:interact}} satisfies
$$
	L_K < \frac{c_0\varphi_0\sigma^2}{16},
$$
then under the condition of fixed $\mathcal G_n$, for $\tilde\rho_t:=\rho(\tilde X_t,\tilde Y_t)$ defined in \eqref{tilde_rho}, one has
\begin{equation}
	\frac{\d}{\d t}\E[\tilde \rho_t|\mathcal G_n] \Le 
	m(\delta) - c\cdot \E[\tilde \rho_t|\mathcal G_n],~~~~
	t\in[t_n,t_{n+1}).
\end{equation}
\end{corollary}
Corollary \ref{theorem:random_rho_rate} can be directly derived from Lemma \ref{lemma:rho_rate} since Assumption \ref{assumption:gamma} holds with $L=2L_K$. Taking the expectation over the filtration $\mathcal G_n$, one obtains
\begin{equation}
	\frac{\d}{\d t}\E[\tilde\rho_t] \Le m(\delta) - 
	c\cdot \E[\tilde \rho_t],~~~~t\in[t_n,t_{n+1}).
\end{equation}
Integrating this equation in the time interval $[t_n,t_{n+1})$ gives
\begin{equation}
	\E[\tilde \rho_{(n+1)\tau}] \Le 
	e^{-c\tau}\E[\tilde \rho_{n\tau}] + 
	\frac{m(\delta)(1-e^{-c\tau})}c,~~~~
	\forall n\Ge0.
	\label{tilde_rho_inequality_1}
\end{equation}
Induction on \eqref{tilde_rho_inequality_1} for the first $n$ time steps gives
\begin{equation}
	\E[\tilde \rho_{n\tau}] \Le 
	e^{-cn\tau}\E[\tilde \rho_0] + 
	\frac{m(\delta)(1-e^{-cn\tau})}c,~~~~
	\forall n\Ge0.
	\label{tilde_rho_inequality}
\end{equation}
Let $\tilde p_t$ be the transition kernel of the RB--IPS \eqref{RBM_IPS}. 
Given the probability distributions $\mu,\nu\in \P_1$,
suppose the initial values $\tilde X_0\sim\mu ,\tilde Y_0\sim \nu$, then $\tilde X_{n\tau}\sim \mu \tilde p_{n\tau},\tilde Y_{n\tau}\sim \nu \tilde p_{n\tau}$. Clearly, \eqref{tilde_rho_inequality} implies
\begin{equation}
	\W_f(\mu \tilde p_{n\tau},\nu \tilde p_{n\tau}) 
	\Le e^{-cn\tau}
	\W_f(\mu,\nu) + 
	\frac{m(\delta)(1-e^{-nc\tau})}{c},~~~~
	\forall n\Ge0.
	\label{tilde_p_inequality}
\end{equation}
A crucial observation of \eqref{tilde_p_inequality} is that the evolution of the distributions $\{\mu \tilde p_{n\tau}\}_{n\Ge0}$ and $\{\nu \tilde p_{n\tau}\}_{n\Ge0}$ does not depend on the coupling scheme, in particular, the free parameter $\delta>0$. Therefore, one may pass the limit $\delta\rightarrow0$ in \eqref{tilde_p_inequality} to obtain
\begin{equation}
	\W_f(\mu \tilde p_{n\tau},\nu \tilde p_{n\tau}) 
	\Le e^{-cn\tau}
	\W_f(\mu,\nu),
	~~~~\forall n\Ge0.
\end{equation}
Concluding the deduction above, we obtain
\begin{theorem}[contractivity]
\label{theorem:RBM_contractivity}
Under Assumptions \textnormal{\ref{assumption:kappa}} and \textnormal{\ref{assumption:interact}}, let $f(r)$ be the distance function defined in Lemma \textnormal{\ref{lemma:distance}}, and $c:=c_0\sigma^2/2$.
Let $\tilde p_t$ be the transition kernel of the RB--IPS \eqref{RBM_IPS}.
If the constant $L_K$ in Assumption $\ref{assumption:interact}$ satisfies
\begin{equation}
	L_K<\frac{c_0\varphi_0\sigma^2}{16},
\end{equation}
then
\begin{equation}
	\W_f(\mu \tilde p_{n\tau},\nu \tilde p_{n\tau}) \Le e^{-cn\tau} \W_f(\mu,\nu),~~~~
	\forall n\Ge0
\end{equation}
for any probability distributions $\mu,\nu\in\P_1$.
\end{theorem}
Theorem \ref{theorem:RBM_contractivity} is a random batch version of Theorem \ref{theorem:IPS_contractivity}. The contraction rate $c$ is a constant of order 1 and does not depend on the number of particles $N$, the batch size $p$ or the time step $\tau$.
\begin{remark}
The continuous time dynamics RB--IPS $\{\tilde X_t\}_{t\Ge0}$ is not a time-homogeneous Markov process, since the random divisions are determined at different time steps. However, $\{\tilde X_{n\tau}\}_{n\Ge0}$ is a time-homogeneous Markov chain, and the transition kernels $\{\tilde p_{n\tau}\}_{n\Ge0}$ forms a semi-group.
\end{remark}
Similar to Corollary \ref{theorem:IPS_ergodicity}, we can prove that for any initial distribution $\nu\in\P_1$, $\nu \tilde p_{n\tau}$ converges to the invariant distribution $\tilde\mu$ exponentially.
\begin{corollary}[contractivity]
\label{theorem:RBM_ergodicity}
Under Assumptions \textnormal{\ref{assumption:kappa}} and \textnormal{\ref{assumption:interact}}, let $f(r)$ be the distance function defined in Lemma \textnormal{\ref{lemma:distance}}, and $c:=c_0\sigma^2/2$.
Let $\tilde p_t$ be the transition kernel of the RB--IPS \eqref{RBM_IPS}, and $\tilde\mu\in\P_1$ be the invariant distribution.
If the constant $L_K$ in Assumption \textnormal{\ref{assumption:interact}} satisfies
$$
	L_K<\frac{c_0\varphi_0\sigma^2}{16},
$$
then we have
\begin{equation}
	\W_f(\nu \tilde p_{n\tau},\tilde\mu) \Le e^{-cn\tau} \W_f(\nu,\tilde\mu),~~~~
	\forall n\Ge0
\end{equation}
for any probability distribution $\nu\in\P_1$.
\end{corollary}
\noindent
The existence of the invariant distribution $\tilde\mu$ will be later proved in Theorem \ref{theorem:existence}.
\section{Error Estimation of Invariant Distributions}
\label{section:error}

In this section we measure the difference between the invariant distributions $\mu,\tilde\mu$ of the IPS \eqref{IPS} and the RB--IPS \eqref{RBM_IPS}. We shall prove the following results:
\begin{enumerate}
\item \textbf{Existence of invariant distributions.} The IPS has an invariant distribution $\mu\in\P_1$, and the RB--IPS has an invariant distribution $\tilde\mu\in\P_1$. This is a direct corollary of the geometric ergodicity proved in Section \ref{section:ergodicity} using the Banach fixed point theorem.
\item \textbf{Strong error estimation in finite time.}
Using the strong error estimation\cite{RBM_error}, for given initial distribution $\nu$, the distance between $\nu p_t$ and $\nu\tilde p_t$ can be bounded by $O(\tau^{\frac12})$, where $p_t,\tilde p_t$ are the transition kernels of the IPS and the RB--IPS respectively.
\item \textbf{Error estimation of invariant distributions.}
Combining the geometric ergodicity and the strong error estimation in finite time, we are able to estimate the difference between the invariant distributions $\mu,\tilde \mu$, using the triangle inequality described in the Introduction.
\end{enumerate}

\subsection{Characterization of invariant distributions}
\label{section:invariant}

We prove the existence of the invariant distributions for the IPS \eqref{IPS} and the RB--IPS \eqref{RBM_IPS} and estimate their first-order moments.
The proof is accomplished by the Banach fixed point theorem on the space $\P_1$ of probability distributions, where we have defined in \eqref{distribution_P1}. Such strategy has previously appeared in \cite{reflection_rate}, which proves the existence of the invariant distribution $\mu$ of the IPS. We extend this strategy to prove the existence of invariant distribution $\tilde\mu$ of the RB--IPS \eqref{RBM_IPS}.

To begin with, we show that the distributions $\nu p_t$ and $\nu \tilde p_t$ always have finite first-order moments.
\begin{lemma}[moment]
\label{lemma:moment}
Under Assumptions $\ref{assumption:kappa}$ and $\ref{assumption:interact}$, there exists a constant $D$ such that if the constant $L_K$ in Assumption $\ref{assumption:interact}$ satisfies
$$
	L_K < \frac{c_0\varphi_0\sigma^2}{16},
$$
then for any probability distribution $\nu\in \P_1$,
\begin{enumerate}
\item[\textnormal{(i)}] $\nu p_t\in\P_1$ for any $t\Ge0$, and
\begin{equation}
\varlimsup_{t\rightarrow\infty}
\int_{\mathbb R^{Nd}} \bigg(\frac1N\sum_{i=1}^N|x^i|\bigg)(\nu p_t)(\d x) \Le D;
\end{equation}
\item[\textnormal{(ii)}] $\nu\tilde p_{n\tau}\in\P_1$ for any $n\Ge0$, and
\begin{equation}
\varlimsup_{n\rightarrow\infty}
\int_{\mathbb R^{Nd}} \bigg(\frac1N\sum_{i=1}^N|x^i|\bigg)(\nu \tilde p_{n\tau})(\d x) \Le D.
\end{equation}
\end{enumerate}
The constant $D$ does not depend on the number of particles $N$, the time step $\tau$, the batch size $p$ and the initial distribution $\nu$.
\end{lemma}
The proof of Lemma \ref{lemma:moment} is left in Appendix. The asymptotic positivity of the function $\kappa(r)$ in Assumption \ref{assumption:kappa} is crucial to bound the moments of $\nu p_t$ and $\nu \tilde p_t$ uniformly in time.
\begin{remark}
As we shall see in strong error estimation, we can also obtain the $\alpha$-th order moment estimation which is uniform in time for a general constant $\alpha\Ge2$.
\end{remark}
Using the contractivity obtained in Section \ref{section:ergodicity}, we derive the existence of the invariant distributions:
\begin{theorem}[existence] 
\label{theorem:existence}
Under Assumptions \textnormal{\ref{assumption:kappa}} and \textnormal{\ref{assumption:interact}}, there exists a constant $D$ such that if the constant $L_K$ in Assumption \textnormal{\ref{assumption:interact}} satisfies
$$
	L_K < \frac{c_0\varphi_0\sigma^2}{16},
$$
then
\begin{enumerate}
\item[\textnormal{(i)}] The Markov process $\{X_t\}_{t\Ge0}$ evolved by the IPS \eqref{IPS} has a unique invariant distribution $\mu\in\P_1$;
\item[\textnormal{(ii)}] The Markov chain $\{\tilde X_{n\tau}\}_{n\Ge0}$ evolved by the RB--IPS \eqref{RBM_IPS} has a unique invariant distribution $\tilde\mu\in\P_1$.
\end{enumerate}
\end{theorem}
The proof below is similar to the proof of Corollary 3 in \cite{reflection_rate}.
\begin{proof}
(i) Note that the Wasserstein distance $\W_f$ is equivalent to the standard $\W_1$-distance
\begin{equation}
	\W_1(\mu,\nu) = \inf_{\gamma\in\Pi(\mu,\nu)}
	\int_{\mathbb R^{Nd}\times\mathbb R^{Nd}}
	\bigg(\frac1N\sum_{i=1}^N|x^i-y^i|\bigg)
	\gamma(\d x\d y).
	\label{W_1_distance}
\end{equation}
From Theorem \ref{theorem:IPS_contractivity}, there exists a constant $C>0$ such that
\begin{equation}
	\W_1(\mu p_t,\nu p_t) \Le Ce^{-ct} \cdot\W_1(\mu,\nu),~~~~
	\forall t\Ge0
\end{equation}
for all distributions $\mu,\nu\in \P_1$. Then there exists a constant $T>0$ such that $q:=Ce^{-cT}<1$ and
\begin{equation}
	\W_1(\mu p_T,\nu p_T) \Le q \cdot \W_1(\mu,\nu).
\end{equation}
Hence the mapping $\nu\mapsto \nu p_T$ is contractive in the complete metric space $\P_1$. From the Banach fixed point theorem, this mapping has a fixed point $\mu_0 \in \P_1$, i.e.,
\begin{equation}
	\mu_0 = \mu_0 p_T.
\end{equation}
Define the distribution
\begin{equation}
	\mu = \frac1T\int_0^T \mu_0 p_s\d s,
\end{equation}
then $\mu$ is a probability distribution in $\mathbb R^{Nd}$ and $\mu\in \P_1$ from Lemma \ref{lemma:moment}.
From the Markov property of the IPS $\{X_t\}_{t\Ge0}$, for any $t\Ge0$ we have
\begin{equation}
	\mu p_t = \frac1T\int_0^T (\mu_0 p_s) p_t\d s = \frac1T\int_0^T \mu_0 p_{s+t}\d s.  
\end{equation}
Since the family of distributions $\{\mu p_t\}_{t\Ge0}$ has the period $T$, we have
\begin{equation}
	 \mu p_t = \frac1T\int_0^T \mu_0 p_s \d s = \mu.
\end{equation}
Therefore, $\mu$ is the invariant distribution of the Markov process $\{X_t\}_{t\Ge0}$. The uniqueness of $\mu$ follows from the contractivity in Theorem \ref{theorem:IPS_contractivity}.\\[6pt]
(ii) For given $\tau>0$, there exists a constant $C>0$ such that
\begin{equation}
	\W_1(\mu\tilde p_{n\tau},\nu \tilde p_{n\tau}) \Le 
	C e^{-nc\tau} \cdot
	\W_1(\mu,\nu),
\end{equation}
then one can choose an integer $N\in\mathbb N$ such that $q= Ce^{-Nc\tau}<1$, and
\begin{equation}
	\W_1(\mu\tilde p_{N\tau},\nu \tilde p_{N\tau}) \Le 
	q\cdot
	\W_1(\mu,\nu),
\end{equation}
so that the mapping $\nu\mapsto \nu \tilde p_{N\tau}$ is contractive. From the Banach fixed point theorem, this  mapping has a fixed point $\tilde\mu_0\in\P_1$, i.e.,
\begin{equation}
	\tilde \mu_0 = \tilde \mu_0 \tilde p_{N\tau}
\end{equation}
Define the distribution
\begin{equation}
	\tilde \mu = \frac1N\sum_{k=0}^{N-1}
	\tilde \mu_0 \tilde p_{k\tau},
\end{equation}
then from Lemma \ref{lemma:moment} $\tilde\mu\in\P_1$.
From the Markov property of the RB--IPS $\{\tilde X_{n\tau}\}_{n\Ge0}$, one has
\begin{equation}
	\tilde \mu \tilde p_{n\tau} = 
	\frac1N\sum_{k=0}^{N-1}
	(\tilde \mu_0 \tilde p_{k\tau}) \tilde p_{n\tau} = 
	\frac1N\sum_{k=0}^{N-1} 
	\tilde \mu_0 \tilde p_{k\tau} = \tilde \mu 
\end{equation}
for any $n\Ge0$. Therefore, $\tilde\mu$ is the invariant distribution of the Markov chain $\{\tilde X_{n\tau}\}_{n\Ge0}$. The uniqueness of $\tilde\mu$ follows from the contractivity in Theorem \ref{theorem:RBM_contractivity}.
\end{proof}
By choosing $\mu$ to be the invariant distribution in Theorem \ref{theorem:IPS_contractivity}, we have
\begin{equation}
	\W_f(\mu, \nu p_t) \Le e^{-ct} \W_f(\mu,\nu),~~~~
	\forall t\Ge0,
\end{equation}
which implies $\nu p_t$ converges to $\mu$ in the sense of the Wasserstein distance $\W_f$. Since $f(r)$ is equivalent to the 
Euclidean norm, Lemma \ref{lemma:moment} directly implies $\mu,\tilde\mu$ have the following first-moment estimation:
\begin{corollary}[moment]
\label{theorem:moment}
Under Assumptions \textnormal{\ref{assumption:kappa}} and \textnormal{\ref{assumption:interact}}, there exist a constant $D$ such that if the constant $L_K$ in Assumption \textnormal{\ref{assumption:interact}} satisfies
$$
	L_K < \frac{c_0\varphi_0\sigma^2}{16},
$$
then
\begin{equation}
	\int_{\mathbb R^{Nd}} \bigg(\frac1N\sum_{i=1}^N|x^i|\bigg)\mu(\d x),
	\int_{\mathbb R^{Nd}} \bigg(\frac1N\sum_{i=1}^N|x^i|\bigg)\tilde\mu(\d x)
	 \Le D,
\end{equation}
where $\mu,\tilde\mu$ are the invariant distributions of the IPS \eqref{IPS} and the RB--IPS \eqref{RBM_IPS} respectively.
The constant $D$ does not depend on the number of particles $N$, the time step $\tau$ or the batch size $p$.
\end{corollary}
Although the invariant distribution $\tilde\mu$  depends on the time step $\tau$, the constant $D$ in Corollary \ref{theorem:moment} is independent of $\tau$. This means the estimate of the first-order moments of $\tilde\mu$ is uniform in $\tau$.
\begin{remark}
The Banach fixed point theorem in the metric space $\P_1$ only implies $\mu,\tilde\mu$ have finite first-order moments, and does not guarantee $\mu,\tilde\mu$ have higher order moments, despite the fact that $\nu p_t$ and $\nu \tilde p_t$ has finite $\alpha$-th order moments for any $\alpha\Ge2$.
\end{remark}

\subsection{Strong error estimation in finite time}

In stochastic analysis, the strong error relates to the trajectory difference between two stochastic processes. Suppose the IPS $X_t$ and the RB--IPS $\tilde X_t$ are driven by the same Wiener process $W_t$ in $\mathbb R^{Nd}$, and the initial state $X_0 = \tilde X_0$ is sampled from the same distribution $\nu\in \P_1$. In other words, $X_t$ and $\tilde X_t$ are coupled in the synchronous coupling scheme. Define the strong error between the trajectories $X_t$ and $\tilde X_t$ by
\begin{equation}
	J(t) = \frac1{2N}
	\sum_{i=1}^N \E|\tilde X_t^i - X_t^i|^2,~~~~
	t\Ge0.
\end{equation}
We aim to estimate $J(t)$ in a finite interval $t\in[0,T]$, and derive the upper bound of $J(t)$ in terms of $\tau$. Except for Assumptions \ref{assumption:kappa} and \ref{assumption:interact}, we additionally require:
\begin{assumption}[bounded]
\label{assumption:bound}
There exists constants $C>0$ and $q\Ge2$ such that
\begin{equation}
	\max\{|b(x)|,|\nabla b(x)|\} \Le C(|x|+1)^q,~~~~
	\forall x\in\mathbb R^d.
\end{equation}
\end{assumption}
\begin{remark}
The requirement $q\Ge2$ in Assumption \textnormal{\ref{assumption:bound}} is merely for technical convenience.
\end{remark}
To analyze $J(t)$ is different time steps, define the filtration
\begin{equation}
	\mathcal F_n = \sigma(\nu,\{W_t\}_{t\Le t_n},\{\D_k\}_{0\Le k\Le n}).
\end{equation}
That is, $\mathcal F_n$ is determined by the initial distribution $\nu$, the Wiener process $W_t$ before $t_n$ and the divisions $\D_k$ in the first $n+1$ time steps. Under the condition of $\mathcal F_n$, the RB--IPS in the time interval $[t_n,t_{n+1})$ is evolved by \eqref{RBM_IPS}. Now we have the following estimate of the $\alpha$-th order moments.
\begin{lemma}[moment]
\label{lemma:moment_alpha}
Under Assumptions \textnormal{\ref{assumption:kappa}} and \textnormal{\ref{assumption:interact}}, for any given constant $\alpha\Ge 2$, there exist positive constants $C,\beta$ depending on $\alpha$ such that for any $i\in\{1,\cdots,N\}$,
\begin{equation}
	\frac{\d}{\d t}
	\E|X_t^i|^\alpha \Le 
	-\beta\cdot \E|X_t^i|^\alpha + C,~~~~
	\forall t\Ge0,
\end{equation}
and
\begin{equation}
	\frac{\d}{\d t}
	\E
	\Big(
	|\tilde X_t^i|^\alpha
	\big| \mathcal F_n
	\Big) \Le -\beta\cdot 
	\E
		\Big(
		|\tilde X_t^i|^\alpha
		\big| \mathcal F_n
		\Big) + C,~~~~
	t\in[t_n,t_{n+1}).
	\label{tilde_X_rate}
\end{equation}
The constants $C,\beta$ do not depend on the number of particles $N$, the time step $\tau$ or the batch size $p$.
\end{lemma}
The proof of Lemma \ref{lemma:moment_alpha} is left in Appendix, and is similar to Lemma 3.3 in \cite{RBM_original}.
The asymptotic positivity of the function $\kappa(r)$ in Assumption \ref{assumption:kappa} is essential to produce the negative coefficient $-\beta$ in \eqref{tilde_X_rate}. By Lemma \ref{lemma:moment_alpha}, we immediately deduce that both $X_t,\tilde X_t$ have finite $\alpha$-th order moments:
\begin{lemma}[moment]
\label{lemma:moment_uniform}
Under Assumptions \textnormal{\ref{assumption:kappa}} and \textnormal{\ref{assumption:interact}}, for any given constant $\alpha\Ge2$, if there exists a constant $M$ such that the initial distribution $\nu$ satisfies
$$
	\max_{1\Le i\Le N}\int_{\mathbb R^{Nd}}
	|x^i|^\alpha \nu(\d x) \Le M,
$$
then there exists a constant $C$ depending on $M,\alpha$ such that
\begin{equation}
	\sup_{t\Ge 0} 
	\E|X_t^i|^\alpha \Le C,~~
		\sup_{t\Ge 0} 
		\E|\tilde X_t^i|^\alpha \Le C.
\end{equation}
The constant $C$ does not depend on the number of particles $N$, the time step $\tau$ or the batch size $p$.
\end{lemma}
\begin{remark}
The constant $C$ in Lemma \ref{lemma:moment_uniform} depends on the moments of the intial distribution $\nu$, hence if one wishes $C$ to be independent of $N$, the moment upper bound $M$ should be also independent of $N$. In particular, if one chooses the initial distribution $\nu$ to be frozen at the origin, then the constant $C$ only relies on $\alpha$.
\end{remark}
The following strong error estimation is exactly the same with the results in \cite{RBM_error}, thus we only present their main theorem here. The detailed proof can be seen at Theorem 3.1 in \cite{RBM_error}.
\begin{theorem}[strong]
\label{theorem:strong_error}
Under Assumptions \textnormal{\ref{assumption:kappa}}, \textnormal{\ref{assumption:interact}} and \textnormal{\ref{assumption:bound}}, if there exists a constant $M$ such that the initial distribution $\nu$ satisfies
$$
	\max_{1\Le i\Le N}\int_{\mathbb R^{Nd}}
	|x^i|^{2q} \nu(\d x) \Le M,
$$
then for any $T>0$, there exists a constant $C$ depending on $T$ and $M$ such that
\begin{equation}
	\sup_{0\Le t\Le T} J(t) \Le C\bigg(
		\frac{\tau}{p-1} + \tau^2
	\bigg).
	\label{estimate_J}
\end{equation}
The constant $C$ does not depend on the number of particles $N$, the time step $\tau$ or the batch size $p$.
\end{theorem}
A slight difference between the statement of Theorem \ref{theorem:strong_error} in this paper and Theorem 3.1 in \cite{RBM_error} is that the latter one does not specify the conditions on the initial distribution $\nu$ explicitly. In fact, finiteness of the $2q$-th order moments is enough to obtain the estimation of $J(t)$ in \eqref{estimate_J}.

Now we can estimate the Wasserstein distance $\W_1(\nu p_t,\nu \tilde p_t)$ using the estimate of $J(t)$, where $\nu$ is the initial distribution, and $p_t,\tilde p_t$ are the transition kernels of the IPS \eqref{IPS} and the RB--IPS \eqref{RBM_IPS}. Recall that $\W_1$-distance between two probability distributions $\mu,\nu\in\P_1$ is defined by 
$$
	\W_1(\mu,\nu) = 
	\inf_{\gamma\in\Pi(\mu,\nu)}
	\int_{\mathbb R^{Nd}\times\mathbb R^{Nd}}
	\bigg(
	\frac1N\sum_{i=1}^N |x^i-y^i|
	\bigg) \gamma(\d x\d y),
$$
hence if we choose $\gamma$ to be the synchronous coupling (driven by the same Wiener process $W_t$), the Wasserstein distance $\W_1(\nu p_t,\nu \tilde p_t)$ can be bounded by
\begin{align*}
	\W_1(\nu p_t,\nu \tilde p_t) & \Le \E\bigg(
	\frac1N\sum_{i=1}^N |X^i-\tilde X_t^i|
	\bigg) \\
	& \Le \sqrt{
		\E
		\bigg(\textit{}
		\frac1N\sum_{i=1}^N |X_t^i-\tilde X_t^i|
		\bigg)^2 
	} \\
	& \Le 
	\sqrt{
	\frac1N\sum_{i=1}^N \E|X_t^i-\tilde X_t^i|^2
	} = \sqrt{2 J(t)},
\end{align*}
that is, $\W_1(\nu p_t,\nu \tilde p_t) \Le \sqrt{2J(t)}$.
Therefore, the estimate of $J(t)$ immediately follows.
\begin{corollary}[Wasserstein]
Under Assumptions \textnormal{\ref{assumption:kappa}}, \textnormal{\ref{assumption:interact}} and \textnormal{\ref{assumption:bound}}, if there exists a constant $M$ such that the initial distribution $\nu$ satisfies
$$
	\max_{1\Le i\Le N}\int_{\mathbb R^{Nd}}
	|x^i|^{2q} \nu(\d x) \Le M,
$$
then for any $T>0$, there exists a constant $C$ depending on $T$ and $M$ such that
\begin{equation}
	\sup_{0\Le t\Le T} \W_1(\nu p_t,\nu \tilde p_t) \Le C\sqrt{
		\frac{\tau}{p-1} + \tau^2
	}.
	\label{estimate_W}
\end{equation}
The constant $C$ does not depend on the number of particles $N$, the time step $\tau$ or the batch size $p$.
\label{theorem:Wasserstein_error}
\end{corollary}
When the batch size $p$ is small, $\sqrt{\tau/(p-1)}$ dominates the Wasserstein error $\W_1(\nu p_t,\nu \tilde p_t)$. In this sense, the Wasserstein error $\W_1(\nu p_t,\nu \tilde p_t)$ has at least half-order convergence in the time step $\tau$.

\subsection{Estimate of $\W_1(\mu,\tilde\mu)$}

Now we estimate $\W_1(\mu,\tilde\mu)$, using the results derived in previous sections.
\begin{theorem}[error]
\label{theorem:overall_error}
Under Assumptions \textnormal{\ref{assumption:kappa}}, \textnormal{\ref{assumption:interact}} and \textnormal{\ref{assumption:bound}}, there exists a constant $C$ such that if the constant $L_K$ in Assumption $\ref{assumption:interact}$ satisfies
$$
	L_K < \frac{c_0\varphi_0\sigma^2}{16},
$$
then the invariant distributions $\mu,\tilde\mu$ of the IPS \eqref{IPS} and the RB--IPS \eqref{RBM_IPS} satisfy
\begin{equation}
	\W_1(\mu,\tilde\mu) \Le C\sqrt{\frac{\tau}{p-1}+\tau^2}.
\end{equation}
The constant $C$ does not depend on the number of particles $N$, the time step $\tau$ or the batch size $p$.
\end{theorem}
The proof of Theorem \ref{theorem:overall_error} is basically the triangle inequality described in the introduction, but with minor difference.
\begin{proof}
For convenience, denote the first-order moment of $\nu\in \P_1$ by
\begin{equation}
	\mathcal M_1(\nu) = \int_{\mathbb R^{Nd}}
	\bigg(
	\frac1N\sum_{i=1}^N |x^i|
	\bigg)\nu(\d x),
\end{equation}
then by Corollary \ref{theorem:moment} $\mathcal M_1(\mu),\mathcal M_1(\tilde\mu)\Le D$. Hence it always holds that
\begin{equation}
	\W_1(\mu,\tilde\mu) \Le \mathcal M_1(\mu) + \mathcal M_1(\tilde\mu) \Le 2D,
\end{equation}
and we may assume $\tau<D$ in the following proof.
Let $\nu_0$ be the distribution in $\mathbb R^{Nd}$ with all the $N$ particles frozen at orgin, then the $2q$-th order moment of $\nu_0$ is $0$. By Lemma \ref{lemma:moment_uniform}, there exists a constant $M$ such that
\begin{equation}
	\sup_{t\Ge0}\bigg\{
	\max_{1\Le i\Le N}
	 \int_{\mathbb R^{Nd}} 
		|x^i|^{2q}
	(\nu_0 p_t)(\d x)\bigg\} < M.
\end{equation}
That is to say, the $2q$-th order moment of $\nu_0 p_t$ is always no greater than $M$.

Instead of directly measuring the distance $\W_1(\mu,\tilde\mu)$, we fix a constant $T>0$ and consider the distance $\W_1(\nu_0 p_T,\tilde\mu)$. By Theorem \ref{theorem:RBM_contractivity}, there exists a constant $C$ such that for any $n\Ge0$,
\begin{align*}
	\W_1(\nu_0 p_T,\tilde\mu) & =
	\W_1(\nu_0 p_T,\tilde\mu\tilde p_{n\tau}) \\
	& \Le \W_1(\nu_0 p_T\tilde p_{n\tau},\tilde\mu\tilde p_{n\tau}) + 
	\W_1(\nu_0 p_T,\nu_0 p_T\tilde p_{n\tau}) \\
	& \Le Ce^{-cn\tau} \W_1(\nu_0 p_T,\tilde\mu) + 
	\W_1(\nu_0 p_T,\nu_0 p_T\tilde p_{n\tau})	
\end{align*}
For given value of $\tau<D$, if one chooses the integer $n$ to be
\begin{equation}
	n = \bigg\lceil
		\frac{\log(2C)}{c\tau}
	\bigg\rceil,
\end{equation}
then $Ce^{-cn\tau}\Le\frac12$ and 
\begin{equation}
	n\tau \Le \bigg(\frac{\log(2C)}{c\tau} + 1\bigg)\tau \Le 
	\frac{\log(2C)}{c} + D,
\end{equation}
hence $n\tau$ has an upper bound. For this chosen $n$ one has
\begin{align*}
	\W_1(\nu_0 p_T,\tilde\mu) & \Le 
	2\cdot \W_1(\nu_0 p_T,\nu_0 p_T\tilde p_{n\tau}) \\
	& \Le 2\cdot \W_1(\nu_0p_T,\nu_0 p_{T}p_{n\tau}) + 2\cdot
	\W_1(\nu_0 p_T p_{n\tau},\nu_0 p_T\tilde p_{n\tau}) \\
	& \Le C e^{-cT} \W_1(\nu_0,\nu_0 p_{n\tau}) + 2\cdot
	\W_1(\nu_0 p_T p_{n\tau},\nu_0 p_T\tilde p_{n\tau}).
\end{align*}
Passing to the limit $T\rightarrow\infty$ gives
\begin{equation}
	\W_1(\mu,\tilde\mu) \Le 2\varlimsup_{T\rightarrow\infty}
	\W_1(\nu_0 p_T p_{n\tau},\nu_0 p_T\tilde p_{n\tau}).
	\label{W1}
\end{equation}
Note that $\nu_0 p_T$ always has finite $2q$-th order moments, hence by Corollary \ref{theorem:Wasserstein_error},
\begin{equation}
	\W_1(\nu_0 p_T p_{n\tau},\nu_0 p_T\tilde p_{n\tau}) \Le 
	C\sqrt{\frac{\tau}{p-1}+\tau^2},~~~~
	\forall T>0,
	\label{W2}
\end{equation}
where the constant $C$ does not depend on $N,\tau,p$ or the choice of $T$. Combining \eqref{W1}\eqref{W2} we obtain the estimate of $\W_1(\mu,\tilde\mu)$:
\begin{equation}
	\W_1(\mu,\tilde\mu) \Le C\sqrt{\frac{\tau}{p-1}+\tau^2},
\end{equation}
which is exactly the result we need.
\end{proof}
\begin{remark}
	We estimate the distance $\W_1(\nu_0 p_T,\tilde\mu)$ instead of $\W_1(\mu,\tilde\mu)$ because it is nontrivial to prove the invariant distributions $\mu,\tilde\mu\in\P_1$ has finite $2q$-th order moments. Therefore, we use a series of distributions $\{\nu_0 p_T\}_{T\Ge0}$ to approximate $\mu$, where the moments of $\nu_0 p_T$ can be easily derived.
\end{remark}
\begin{remark}
In this framework, the order of accuracy in the estimation of $\W_1(\mu,\tilde\mu)$ cannot be greater than the order of the strong error. It is still an open question whether it is possible to apply the weak error estimation instead of the strong one in this framework to estimate the difference between $\mu$ and $\tilde\mu$. In this work, the main difficulty is that we can only derive the geometric ergodicity in the sense of the Wasserstein distance, which is stronger than the weak error.
\end{remark}
\begin{appendices}
\section{Proof of main results}
\label{section:appendix}
\textbf{Proof of Lemma \ref{lemma:distance}}
Under Assumption \ref{assumption:kappa}, define the constants $R_0,R_1\Ge 0$ by
\begin{align}
	R_0 & := \inf\{R\Ge0: \kappa(r) \Ge0,\forall r\Ge R\},
	\\
	R_1 & := \inf\{R\Ge R_0: \kappa(r) R(R-R_0) \Ge 16,
	\forall r\Ge R\}.
\end{align}
The existence of $R_0,R_1$ is guaranteed by the asymptotic positivity of $\kappa(r)$. Also, one has $\kappa(r)\Ge 0$ for $r\Ge R_0$ and $\kappa(r) R_1(R_1-R_0) \Ge 16$ for $r\Ge R_1$. Given the function $\kappa(r)$, define the auxiliary functions $\varphi(r),\Phi(r),g(r)$ by
\begin{equation}
	\varphi(r) =  \exp\bigg(
		-\frac14 \int_0^r 
		s\kappa(s)^- \d s
	\bigg),~~~~
	\Phi(r) = \int_0^r \varphi(s)\d s,
	\label{proof:distance_f_1}
\end{equation}
\begin{equation}
	g(r) =
	\left\{
	\begin{aligned}
	& 1 - \frac12
	\int_0^{r}
	\frac{\Phi(s)}{\varphi(s)}\d s 
	\bigg/
	\int_0^{R_1}
		\frac{\Phi(s)}{\varphi(s)}\d s, && r\Le R_1 \\
	& \frac12-\frac{\eta(r-R_1)}{1+4\eta(r-R_1)}, && r>R_1
	\end{aligned}
	\right.
	\label{proof:distance_f_2}
\end{equation}
where $x^- = -\min\{x,0\}$ is the negative part of $x\in\mathbb R$ and the constant $\eta>0$ is defined by
\begin{equation}
	\eta = -g'(R_1) = \frac12\frac{\Phi(R_1)}{\varphi(R_1)} \bigg/
	\int_0^{R_1} \frac{\Phi(s)}{\varphi(s)}\d s.
	\label{proof:distance_f_3}
\end{equation}
\eqref{proof:distance_f_3} ensures that $g(r)$ is differentiable at $r=R_1$.
Finally, the distance function $f(r)$ is defined as
\begin{equation}
	f(r) = \int_0^r \varphi(s)g(s)\d s.
	\label{proof:distance_f_4}
\end{equation}
The only difference between Eq. \eqref{proof:distance_f_1}-\eqref{proof:distance_f_4} and the construction of $f(r)$ in \cite{reflection_original} is the definition of $g(r)$ for $r>R_1$. In our choice, $g(r)$ is differentiable at $r=R_1$ so that $f(r)$ is always twice differentiable, while in the original proof $f(r)\in C^1$ and $f'(r)$ is absolutely continuous.

From Eq. \eqref{proof:distance_f_1}-\eqref{proof:distance_f_4}, it is easy to verify the following properties of the functions $f(r),\varphi(r),\Phi(r),g(r)$:
\begin{enumerate}
\item $0<\varphi(r)\Le1$, $\frac14\Le g(r)\Le1$. $\varphi(0) = g(0) = 1$. $\Phi(0) = 0$.
\item The derivatives of $\varphi$ and $g$ are given by
\begin{equation}
	\varphi'(r) = -\frac14r\kappa(r)^- \varphi(r),~~~~~
	g'(r) = -\frac12 \frac{\Phi(r)}{\varphi(r)}
	\bigg/ \int_0^{R_1} \frac{\Phi(s)}{\varphi(s)}\d s,~~~~
	0\Le r\Le R_1.
\end{equation}
Hence $\varphi'(0) = g'(0) = 0$ and  $\varphi'(r)\Le0,g'(r)\Le0$ for all $r\Ge0$.
\item The second derivative of $f(r)$ is given by
\begin{equation}
	f''(r) = \varphi(r) g'(r) + \varphi'(r) g(r) \Le0,
\end{equation}
which implies $f(r)$ is concave for all $r\Ge0$.
\item When $r>R_0$,
\begin{equation}
	\varphi(r) \equiv \varphi_0 := \exp\bigg(
	-\frac14\int_0^{R_0}s\kappa(s)^-\d s
	\bigg),
\end{equation}
Since $\varphi(r)\Ge \varphi_0$ and $g(r)\Ge\frac14$ for all $r\Ge0$, one obtains the estimate
\begin{equation}
	f'(r) = \varphi(r) g(r) \Ge \frac{\varphi_0}4.
\end{equation}
which implies $f(r)\Ge \frac{\varphi_0}4r$ for all $r\Ge0$.
\item Since $g(r)\Le1$, 
\begin{equation}
	\Phi(r) = \int_0^r \varphi(s)\d s \Ge 
	\int_0^r \varphi(s) g(s)\d s = f(r).
\end{equation}
From $\Phi''(r) = \varphi'(r) \Le 0$, $\Phi(r)$ is also concave for $r\in[0,+\infty)$.
\end{enumerate}
Now one can prove the inequality \eqref{f_inequality} with the constant $c_0$ defined by
\begin{equation}
	\frac1{c_0} = \int_0^{R_1} \frac{\Phi(s)}{\varphi(s)}\d s.
	\label{proof:c_0}
\end{equation}
\begin{enumerate}
\item When $r\Le R_1$, using $f(r)\Le \Phi(r)$,
\begin{align*}
	f''(r) & = \varphi'(r) g(r) + \varphi(r) g'(r) \\
	& = -\frac14 r\kappa(r)^- \varphi(r) g(r) 
	- \frac12 \Phi(r) \bigg/
	\int_0^{R_1} \frac{\Phi(s)}{\varphi(s)}\d s \\
	& \Le \frac14 r \kappa(r) f'(r) - 
	\frac12 f(r) \bigg/
	\int_0^{R_1} \frac{\Phi(s)}{\varphi(s)}\d s,
\end{align*}
hence \eqref{f_inequality} holds with $c_0$ defined in \eqref{proof:c_0}.
\item When $r>R_1$, $f'(r) \Ge \varphi_0/4$, $f''(r)\Le 0$. Hence by the definition of $R_1$ and the concavity of $\Phi(r)$ with $\Phi(0) = 0$, one has
\begin{equation}
	f''(r) - \frac14r\kappa(r) f'(r) \Le 
	-\frac1{16} r\kappa(r) \varphi_0 
	\Le -\frac{\varphi_0}{R_1 - R_0} 
		\frac{r}{R_1} \\
		\Le -\frac{\varphi_0}{R_1 - R_0} \frac{\Phi(r)}{\Phi(R_1)}.
	\label{proof:verify_f_1}
\end{equation}
Since $\varphi(r) \equiv \varphi_0$ for $r\Ge R_0$, $\Phi(r)$ is linear in $r$, i.e.,
\begin{equation}
	\Phi(r) = \Phi(R_0) + (r-R_0)\varphi_0,~~~~r\Ge R_0.
\end{equation}
In particular, $\Phi(R_1) = \Phi(R_0) + (R_1-R_0)\varphi_0$, hence
\begin{equation}
	\int_{R_0}^{R_1} 
	\frac{\Phi(s)}{\varphi(s)}\d s = 
	\frac{\Phi(R_0)}{\varphi_0} (R_1 - R_0) + \frac12(R_1 - R_0)^2 \Ge \frac12(R_1 - R_0) \frac{\Phi(R_1)}{\varphi_0}.
	\label{proof:verify_f_2}
\end{equation} 
Combining \eqref{proof:verify_f_1}\eqref{proof:verify_f_2} one obtains
\begin{equation}
	f''(r) - \frac14 r\kappa(r) f'(r) \Le -\frac12\Phi(r)\bigg/
	\int_{R_0}^{R_1} 
	\frac{\Phi(s)}{\varphi(s)}\d s \Le 
	-\frac12 f(r) \bigg/
	\int_{0}^{R_1}
	\frac{\Phi(s)}{\varphi(s)}\d s,
\end{equation}
hence \eqref{f_inequality} holds with $c_0$ defined in \eqref{proof:c_0}.
\end{enumerate}
It is easy to see $\frac{\varphi_0}4r \Le f(r) \Le r$ for all $r\Ge0$.\\[12pt]
\textbf{Proof of Lemma \ref{lemma:sum_f_inequality}}
Using $b^i(x) = b(x^i) + \gamma^i(x)$, the
LHS of \eqref{sum_f_inequality} is written as $I=I_1+I_2+I_3$,
\begin{align*}
	I_1 & = \sum_{i=1}^N (r^i)^{-1} Z^i \cdot (b(X^i) - b(Y^i))f'(r^i), \\
	I_2 & = \sum_{i=1}^N (r^i)^{-1} Z^i \cdot (\gamma^i(X) - \gamma^i(Y)) f'(r^i), \\
	I_3 & = 2\sigma^2\sum_{i=1}^N \lambda^2(Z^i) f''(r^i).
\end{align*}
Now we estimate $I_1,I_2,I_3$ respectively.
\begin{itemize}
\item Estimate $I_1$: By the definition of $\kappa(r)$ in \eqref{kappa},
\begin{equation}
	I_1 \Le -\frac{\sigma^2}2\sum_{i=1}^N 
	r^i \kappa(r^i)f'(r^i).
	\label{proof:sum_f_inequality_I_1}
\end{equation}
\item Estimate $I_2$: Using the Lipschitz condition in Assumption \ref{assumption:gamma} and $f(r) \Ge \varphi_0r/4$,
\begin{equation}
	I_2 \Le \sum_{i=1}^N |\gamma^i(X) - \gamma^i(Y)| \Le L\sum_{i=1}^N r^i \Le 
	\frac{4L}{\varphi_0}\sum_{i=1}^N f(r^i).
	\label{proof:sum_f_inequality_I_2}
\end{equation}
\item Estimate $I_3$: Using the estimation of $f''(r)$ in \eqref{f_inequality},
\begin{align}
	I_3 & \Le \frac{\sigma^2}2\sum_{i=1}^N
		r^i \kappa(r^i) \lambda^2(Z^i)f'(r^i)  
	- c_0\sigma^2\sum_{i=1}^N \lambda^2(Z^i) f(r^i) \notag \\
	& = \frac{\sigma^2}2\sum_{i=1}^N r^i\kappa(r^i)f'(r^i)  
		- c_0\sigma^2\sum_{i=1}^N f(r^i) \notag \\
	& ~~~~ \underbrace{-\frac{\sigma^2}2\sum_{i=1}^N 
	r^i\kappa(r^i) (1-\lambda^2(Z^i))f'(r^i)}_{I_{31}} + \underbrace{c_0\sigma^2\sum_{i=1}^N 
			(1-\lambda^2(Z^i)) f(r^i)}_{I_{32}}.
	\label{proof:sum_f_inequality_I_3_split}
\end{align}
We estimate $I_{31}$ and $I_{32}$ in \eqref{proof:sum_f_inequality_I_3_split}.
\begin{itemize}
\item Estimate $I_{31}$:
Note that $1-\lambda^2(Z^i) = 0$ if $r_i\Ge\delta$, thus
\begin{align*}
	I_{31} & = -\frac{\sigma^2}2\sum_{i=1}^N 
	r^i \kappa(r^i) (1-\lambda^2(Z^i))f'(r^i)  \\
	& \Le
	\frac{\sigma^2}2\sum_{i:r^i<\delta}
	r^i \kappa(r^i)^-f'(r^i) \\
	& \Le \frac{\sigma^2}2\sum_{i:r^i<\delta}
	r^i\kappa(r^i)^- \\
	& \Le \frac{N\sigma^2}2\sup_{r<\delta}
	\Big(r\kappa(r)^-\Big).
\end{align*}
\item Estimate $I_{32}$: In a similar way, using $f(r)\Le r$ one obtains
\begin{align*}
	I_{32} & =  c_0\sigma^2\sum_{i=1}^N  (1-\lambda^2(Z^i)) f(r^i) \\
	& \Le c_0\sigma^2\sum_{i:r^i<\delta} f(r^i) \\
	& \Le c_0N\sigma^2\delta.
\end{align*}
\end{itemize}
From the definition of $m(\delta)$ in \eqref{m_delta}, one obtains the estimate of $I_3$:
\begin{equation}
	I_3 \Le \frac{\sigma^2}2\sum_{i=1}^N
		r^i \kappa(r^i)f'(r^i) -
		c_0\sigma^2\sum_{i=1}^N  f(r^i) + Nm(\delta).
	\label{proof:sum_f_inequality_I_3}
\end{equation}
\end{itemize}
Summation over the estimates
\eqref{proof:sum_f_inequality_I_1}\eqref{proof:sum_f_inequality_I_2}\eqref{proof:sum_f_inequality_I_3}
of $I_1,I_2,I_3$ gives
\begin{equation}
	I \Le - \bigg(c_0\sigma^2-\frac{4L}{\varphi_0}\bigg)\sum_{i=1}^N f(r^i) + Nm(\delta).
\end{equation}
When the Lipschitz constant $L<c_0\varphi_0\sigma^2/8$, one has
\begin{equation}
	I \Le -\frac{c_0\sigma^2}2 \sum_{i=1}^N f(r^i) + Nm(\delta) = Nm(\delta) - c\sum_{i=1}^N f(r^i),
\end{equation}
which is exactly the result we need.\\[12pt]
\textbf{Proof of Lemma \ref{lemma:moment}}
Consider the stochastic processes $X_t$ and $\tilde X_t$ evolved by the IPS \eqref{IPS} and the RB--IPS \eqref{RBM_IPS} respectively, with the initial distribution $\nu\in\P_1$. For convenience, we unify \eqref{IPS}\eqref{RBM_IPS} in the form of the product model \eqref{product_model}.\\[6pt]
(i) By choosing a smooth function $f(x) = \sqrt{|x|^2+1}$, 
each $f(X_t^i)$ satisifies the SDE
\begin{equation}
	\d f(X_t^i) = 
		b^i(X_t)\cdot \nabla f(X_t^i)\d t + 
		\frac{\sigma^2}2\Delta f(X_t^i)\d t + 
		\nabla f(X_t^i)\cdot \sigma\d W_t^i,
\end{equation}
where $\Delta = \nabla \cdot \nabla$ is the Laplacian operator in $\mathbb R^d$. Taking the expectation, one obtains
\begin{equation}
	\frac{\d}{\d t} \E[f(X_t^i)] = 
	\E\bigg(
		b^i(X_t)\cdot \nabla f(X_t^i)
		+\frac{\sigma^2}2\Delta f(X_t^i)
	\bigg).
\end{equation}
Note that the 1st, 2nd derivatives of $f(x)$ and the perturbation $\gamma^i(x)$ are uniformly bounded (we have assumed $K(\cdot)$ to be bounded in Assumption \ref{assumption:interact}), for each $i\in\{1,\cdots,N\}$ there is
\begin{equation}
	\frac{\d}{\d t}\E[f(X_t^i)] \Le 
	\E\Big(
	b(X_t^i)\cdot \nabla f(X_t^i)
	\Big)+C =
	\E\bigg(
	\frac{b(X_t^i)\cdot X_t^i}{\sqrt{|X_t^i|^2+1}}
	\bigg)+C.
	\label{proof:ef}
\end{equation}
Under Assumption \ref{assumption:kappa},
we claim that there exists constants $C,\beta>0$ such that
\begin{equation}
	\frac{x\cdot b(x)}{\sqrt{|x|^2+1}} \Le 
	C - \beta \sqrt{|x|^2+1},~~~~\forall x\in\mathbb R^d.
	\label{proof:C_beta}
\end{equation}
In fact, from $\kappa(r)^- = 0$ for $r\Ge R_0$, one has
\begin{align*}
	x\cdot b(x) & \Le x\cdot b(0) - \frac{\sigma^2}2 \kappa(r) |x|^2 \\
	& = x\cdot b(0) - 
	\frac{\sigma^2}2 \kappa(r)^+ |x|^2 + 
	\frac{\sigma^2}2 \kappa(r)^- |x|^2 \\
	& \Le x\cdot b(0) - \frac{\sigma^2}2 \kappa(r)^+ 
	|x|^2 + \frac{\sigma^2R_0}2 \kappa(r)^- |x| \\
	& \Le C|x| - \frac{\sigma^2}2 \kappa(r)^+|x|^2,
\end{align*}
where $x^+=\max\{x,0\}$ denotes the positive part of $x\in\mathbb R$. Thus \eqref{proof:C_beta} holds true.
Combining \eqref{proof:C_beta}\eqref{proof:ef} yields
\begin{equation}
	\frac{\d}{\d t}\E[f(X_t^i)] \Le C - \beta\cdot \E[f(X_t^i)],~~~~
	i=1,\cdots,N.
	\label{proof:efi}
\end{equation}
For the IPS $\{X_t\}_{t\Ge0}$, define
\begin{equation}
	m(t) = \E\bigg(\frac1N \sum_{i=1}^N f(X_t^i)\bigg),
	~~~~\forall t\Ge0.
\end{equation}
Since $\nu$ is the initial distribution of $X_0$,
clearly $m(t)$ is an upper bound of
\begin{equation}
	\int_{\mathbb R^{Nd}}
	\bigg(
	\frac1N\sum_{i=1}^N |x^i|
	\bigg) (\nu p_t) (\d x).
\end{equation}
Summation over $i\in\{1,\cdots,N\}$ in \eqref{proof:efi} gives
\begin{equation}
	m'(t) \Le C - \beta\cdot m(t).
	\label{proof:m_rate}
\end{equation}
Hence $m(t)$ is finite for all $t\Ge0$, and by Gronwall's inequality,
\begin{equation}
	\varlimsup_{t\rightarrow\infty}
	\int_{\mathbb R^{Nd}} \bigg(\frac1N\sum_{i=1}^N
	|x^i|\bigg)(\nu p_t)(\d x) 
	\Le
	\varlimsup_{t\rightarrow\infty} m(t) 
	\Le \frac{C}{\beta}.
\end{equation}
Now one may just take $D > C/\beta$ in Lemma \ref{lemma:moment}.\\[6pt]
(ii) The moment estimate for the RB--IPS can be derived in a similar way. For convenience, define the filtration $\mathcal F_n$ by
\begin{equation}
	\mathcal F_n = \sigma(\nu,
	\{W_s\}_{0\Le s\Le t_n},
	\{\D_k\}_{0\Le k\Le n}).
	\label{proof:filtration_F}
\end{equation}
That is, $\mathcal F_n$ is determined by the initial distribution $\nu$ of $\tilde X_0$, the Wiener process $W_t$ before time $t_n$, and the divisions in the first $n+1$ time steps.
For the RB--IPS $\{\tilde X_t\}_{t\Ge0}$, define
\begin{equation}
	\tilde m(t) = 
	\mathbb E
	\bigg(
	\frac1N\sum_{i=1}^N f(\tilde X_t^i)
	\bigg),~~~~\forall t\Ge0.
	\label{proof:m}
\end{equation}
Under the condition of the filtration $\mathcal F_n$, define
\begin{equation}
	\tilde m(t|\mathcal F_n) := 
	\E\bigg(\frac1N \sum_{i=1}^N f(\tilde X_t^i)
	\bigg| \mathcal F_n
	\bigg),
	~~~~t\in[t_n,t_{n+1}).
\end{equation}
With fixed division $\D_n$ of the index set $\{1,\cdots,N\}$, $\tilde X_t$ in the time interval $[t_n,t_{n+1})$ is evolved by \eqref{RBM_IPS}, and Assumption \ref{assumption:gamma} still holds true with the constant $L=2L_K$. Therefore,
similarly with (\ref{proof:m_rate}), one obtains
\begin{equation}
	\tilde m'(t|\mathcal F_n) \Le C - \beta \cdot \tilde m(t|\mathcal F_n),
	~~~~\forall t\in[t_n,t_{n+1}).
	\label{proof:mF}
\end{equation} 
Taking the expectation over $\mathcal F_n$ in \eqref{proof:mF} gives
\begin{equation}
	\tilde m'(t) \Le C - \beta \cdot \tilde m(t),~~~~
	t\in[t_n,t_{n+1}).
	\label{proof:tilde_mC}
\end{equation}
Integrating \eqref{proof:tilde_mC} in the time interval $[t_n,t_{n+1})$ gives
\begin{equation}
	\tilde m((n+1)\tau) \Le e^{-\beta\tau}
	\tilde m(n\tau) + 
	\frac{C}{\beta}(1-e^{-\beta\tau}),~~~~
	\forall n\Ge0.
\end{equation}
Hence $\tilde m(t)$ is finite for all integers $n\Ge0$, and
by Gronwall's inequality,
\begin{equation}
	\varlimsup_{n\rightarrow\infty} 
	\tilde m(n\tau) \Le \frac{C}{\beta}.
\end{equation}
Now one may just take $D > C/\beta$ in Lemma \ref{lemma:moment}.\\[12pt]
\textbf{Proof of Lemma \ref{lemma:moment_alpha}}
We first estimate $\E|X_t^i|^\alpha$ for the IPS. 
By It\^o calculus,
\begin{equation}
	\frac{\d}{\d t}\E|X_t^i|^\alpha = 
		\alpha\cdot\E \Big\{
		|X_t^i|^{\alpha-2} \Big(
			X_t^i\cdot b(X_t^i) + 
			X_t^i\cdot \gamma^i(X_t)
		\Big) \Big\} + 
		\frac12\alpha(\alpha+d-2)\sigma^2\E|X_t^i|^{\alpha-2},
	\label{proof:X_alpha_rate_1}
\end{equation}
where the perturbation $\gamma^i(x)$ is given by \eqref{gamma_IPS}. By the definition of $\kappa(r)$, one has
\begin{equation}
	-x\cdot (b(x)-b(0)) \Ge \frac{\sigma^2}2\kappa(|x|)|x|^2,~~~~
	\forall x\in\mathbb R^d.
\end{equation}
Hence the drift force part in \eqref{proof:X_alpha_rate_1} is bounded by
\begin{equation}
	|X_t^i|^{\alpha-2} X_t^i\cdot b(X_t^i) 
	\Le C|X_t^i|^{\alpha-1} - \frac{\sigma^2}2\kappa(|X_t^i|)
	|X_t^i|^\alpha.
	\label{proof:drift_estimate}
\end{equation}
Since $\gamma^i(x)$ is uniformly bounded according to Assumption \ref{assumption:interact}, the perturbation part in \eqref{proof:X_alpha_rate_1} is bounded by
\begin{equation}
	|X_t^i|^{\alpha-2}
	X_t^i\cdot \gamma^i(X_t)
	\Le C|X_t^i|^{\alpha-1}.
	\label{proof:perturbation_estimate}
\end{equation}
Combining \eqref{proof:drift_estimate}\eqref{proof:perturbation_estimate}, from \eqref{proof:X_alpha_rate_1} one deduces that
\begin{equation}
	\frac{\d}{\d t}\E|X_t^i|^\alpha \Le
	-\frac{\alpha\sigma^2}{2} 
	\E\big(\kappa(|X_t^i|) 
	|X_t^i|^\alpha\big) + C\big(\E|X_t^i|^{\alpha-1}+\E|X_t^i|^{\alpha-2}\big)
\end{equation} 
Since $\kappa(r) \Ge\delta$ for $r\Ge R_0$ and $\kappa(r)$ has a lower bound for $r>0$, one has
\begin{equation}
	-\kappa(r) r^\alpha = (\delta -\kappa(r))
	r^\alpha - \delta r^\alpha \Le 
	C - \delta r^\alpha,~~~~
	\forall r\Ge0,
\end{equation}
which implies
\begin{align*}
	-\E\big(\kappa(|X_t^i|) 
	|X_t^i|^\alpha\big) & \Le
	C -\delta \cdot\E|X_t^i|^\alpha
\end{align*}
Therefore by choosing $c = \alpha\sigma^2\delta/2$, one has
\begin{equation}
	\frac{\d}{\d t}\E|X_t^i|^\alpha \Le 
	-c\cdot\E|X_t^i|^\alpha + C(\E|X_t^i|^{\alpha-1} + \E|X_t^i|^{\alpha-2}+1)
	\label{proof:X_alpha_rate_2}
\end{equation}
Using interpolation inequality, $\E|X_t^i|^{\alpha-1}$ and $\E|X_t^i|^{\alpha-2}$ can be bounded by $\E|X^i|^{\alpha}$ plus constant. Therefore, \eqref{proof:X_alpha_rate_2} implies
\begin{equation}
	\frac{\d}{\d t}\E|X_t^i|^\alpha \Le 
	-\frac{c}2\cdot\E|X_t^i|^\alpha + C
\end{equation}
for some constant $C$, which is exactly the result we need. For the RB--IPS \eqref{RBM_IPS}, the perturbation $\gamma^i(x)$ given by \eqref{gamma_RBM} is still bounded by $2L_K$, thus the proof above still holds true for the RB--IPS in the time interval $[t_n,t_{n+1})$ under the condition of $\mathcal F_n$.
\end{appendices}
\section*{Acknowledgements}
S. Jin was partially supported by the NSFC grant No.~12031013, the Shanghai Municipal Science and Technology Major Project, and Science and Technology
Commission of Shanghai Municipality grant No. 20JC1414100, (2021SHZDZX0102). 
L. Li was partially sponsored by the Strategic Priority Research Program of Chinese Academy of Sciences, Grant No.~XDA25010403, and NSFC 11901389, 12031013.
Z. Zhou was partially supported by the National Key R\&D Program of China, Project Number 2021YFA1001200 and the NSFC grant, No.~12171013.

\bibliography{reference.bib}
\bibliographystyle{unsrt}
\end{document}